\documentclass[a4paper,10pt]{amsproc}
\usepackage{amsfonts, amsmath, amssymb, graphicx, mathrsfs, tikz, xcolor, hyperref}
\usepackage[a4paper, total={6in, 10in}]{geometry}
\tikzset{Bullet/.style={fill=black,draw,color=#1,circle,minimum size=2pt,scale=0.45}}
\usetikzlibrary{shapes.geometric}
\usetikzlibrary{positioning, tikzmark}
\usepackage[all]{xy}
\theoremstyle{plain}

\newtheorem{theorem}{Theorem}[section]
\newtheorem{lemma}[theorem]{Lemma}

\theoremstyle{definition}
\newtheorem{definition}[theorem]{Definition}

\newtheorem{observation}[theorem]{Observation}
\newtheorem{remark}[theorem]{Remark}
\newtheorem{counter example}[theorem]{Counter Example}

\newtheorem{corollary}[theorem]{Corollary}
\newtheorem{example}[theorem]{Example}

\numberwithin{equation}{section}

\begin{document}
	\Large{\title{Annihilator graph of the ring $C_\mathscr{P}(X)$}
	\author[P. Nandi]{Pratip Nandi}
		\address{Department of Pure Mathematics, University of Calcutta, 35, Ballygunge Circular Road, Kolkata 700019, West Bengal, India}
		\email{pratipnandi10@gmail.com}
	\author[S.K. Acharyya]{Sudip Kumar Acharyya}
		\address{Department of Pure Mathematics, University of Calcutta, 35, Ballygunge Circular Road, Kolkata 700019, West Bengal, India}
		\email{sdpacharyya@gmail.com}

		\author[A. Deb Ray]{Atasi Deb Ray}
		\address{Department of Pure Mathematics, University of Calcutta, 35, Ballygunge Circular Road, Kolkata 700019, West Bengal, India}
		\email{debrayatasi@gmail.com}
		
		\thanks{The first author thanks the CSIR, New Delhi – 110001, India, for financial support}

\begin{abstract}
In this article, we introduce the annihilator graph of the ring $C_\mathscr{P}(X)$, denoted by $AG(C_\mathscr{P}(X))$ and observe the effect of the underlying Tychonoff space $X$ on various graph properties of $AG(C_\mathscr{P}(X))$. $AG(C_\mathscr{P}(X))$, in general, lies between the zero divisor graph and weakly zero divisor graph of  $C_\mathscr{P}(X)$ and it is proved that these three graphs coincide if and only if the cardinality of the set of all $\mathscr{P}$-points, $X_\mathscr{P}$ is $\leq 2$. Identifying a suitable induced subgraph of $AG(C_\mathscr{P}(X))$, called $G(C_\mathscr{P}(X))$, we establish that both $AG(C_\mathscr{P}(X))$ and $G(C_\mathscr{P}(X))$ share similar graph theoretic properties and have the same values for the parameters, e.g., diameter, eccentricity, girth, radius, chromatic number and clique number. By choosing the ring $C_\mathscr{P}(X)$ where $\mathscr{P}$ is the ideal of all finite subsets of $X$ such that $X_\mathscr{P}$ is finite, we formulate an algorithm for coloring the vertices of $G(C_\mathscr{P}(X))$ and thereby get the chromatic number of $AG(C_\mathscr{P}(X))$. This exhibits an instance of coloring infinite graphs by just a finite number of colors. We show that any graph isomorphism $\psi : AG(C_\mathscr{P}(X)) \rightarrow AG(C_\mathscr{Q}(Y))$ maps $G(C_\mathscr{P}(X))$ isomorphically onto $G(C_\mathscr{Q}(Y))$ as a graph and a graph isomorphism $\phi : G(C_\mathscr{P}(X)) \rightarrow G(C_\mathscr{Q}(Y))$ can be extended to a graph isomorphism $\psi : AG(C_\mathscr{P}(X)) \rightarrow AG(C_\mathscr{Q}(Y))$ under a mild restriction on the function $\phi$. Finally, we show that atleast for the rings $C_\mathscr{P}(X)$ with finitely many $\mathscr{P}$-points, so far as the graph properties are concerned, the induced subgraph $G(C_\mathscr{P}(X))$ is a good substitute for $AG(C_\mathscr{P}(X)$. 
\end{abstract}


\maketitle

\section{Introduction}
Let $C(X)$ be the ring of all real valued continuous functions defined on a Tychonoff space $X$ and  $\mathscr{P}$, an ideal of closed sets in $X$ in the following sense: if $A\in\mathscr{P}$ and $B\in\mathscr{P}$, then $A\cup B\in\mathscr{P}$ and if $A\in\mathscr{P}$ and $C\subset A$ with $C$, closed in $X$, then $C\in\mathscr{P}$. In \cite{Ach}, we have studied the zero divisor graph of the ring $C_\mathscr{P}(X) = \{f\in C(X) : cl_X(X\setminus Z(f))\in\mathscr{P}\}$, where $Z(f)=\{x\in X:f(x)=0\}$ is the zero set of $f$. In this article, we introduce the annihilator graph of the same ring which we denote by $AG(C_\mathscr{P}(X))$ and observe how the topology of the underlying space helps to interpret its graph properties and vice-versa. Though Badawi investigated the annihilator graph of a commutative ring in \cite{Bad}  and achieved the values of typical graph parameters through the properties of the ring, our study of annihilator graph of $C_\mathscr{P}(X)$ reveals the interactions among the properties of the ring $C_\mathscr{P}(X)$, the graph $AG(C_\mathscr{P}(X))$ and the topological behaviour of the space $X$. \\

It is well known that the annihilator graph of any commutative ring $R$ contains, as a subgraph, the zero divisor graph of $R$ \cite{Bad}. We show here that the annihilator graph of $C_\mathscr{P}(X)$ is contained in its weakly zero divisor graph. We also establish that the cardinality of the set of all $\mathscr{P}$ points of the topological space $X$ completely determines when and only when all the three graphs coincide; where the $\mathscr{P}$ points are those $x\in X$ such that $cl_X V \in \mathscr{P}$, for some neighborhood $V$ of $x$. In section 3, certain graph parameters, e.g., the diameter, radius, girth, length of the smallest cycle containing two given vertices, are calculated and condition(s) under which the annihilator graph of $C_\mathscr{P}(X)$ is triangulated, hypertriangulated and uniquely complemented, are determined. \\

A subgraph $H$ of a given graph $G$ is called an induced subgraph if two vertices $a, b$ of $H$ are adjacent in $G$ whenever they are adjacent in $H$. In Section 4 of this article we identify an induced subgraph of $AG(C_\mathscr{P}(X))$, denoted by $G(C_\mathscr{P}(X))$ which exhibit a similar behaviour as a graph, so far as typical graph properties are concerned. For example, both have same diameter, girth, radius, eccentricity, clique and chromatic number. At times, handling $G(C_\mathscr{P}(X))$ comes out more convenient, as it has less number of edges than those of $AG(C_\mathscr{P}(X))$. In Section 5, knowing that there are only finitely many isolated points in $X$, we device an algorithm for coloring the vertices of $G(C_F(X))$ (where $C_F(X)$ is the ring obtained by choosing $\mathscr{P}$ as the ideal of all finite subsets of $X$) and show how this coloring of $G(C_{F}(X))$ determines a coloring for $AG(C_F(X))$, although it is an infinite graph.\\

The last section of this article discusses how far the role of $G(C_\mathscr{P}(X))$ can replace the role of $AG(C_\mathscr{P}(X))$ as a graph. We answer in affirmative that any graph isomorphism $AG(C_\mathscr{P}(X)) \rightarrow AG(C_\mathscr{Q}(Y))$ maps $G(C_\mathscr{P}(X))$ isomorphically onto $G(C_\mathscr{Q}(Y))$. Perhaps it is more intriguing whether any graph isomorphism $G(C_\mathscr{P}(X)) \rightarrow G(C_\mathscr{Q}(Y))$ can be extended to a graph isomorphism $AG(C_\mathscr{P}(X)) \rightarrow AG(C_\mathscr{Q}(Y))$. We have shown in Section 6 that under some mild conditions, this is true. We also show that $C_\mathscr{P}(X)$ with $|X_\mathscr{P}|$ (= cardinality of $X_\mathscr{P}$) $= $ finite, is indeed an example for which the graph properties of $G(C_\mathscr{P}(X))$ (and hence, $AG(C_\mathscr{P}(X))$) completely determines the ring properties of $C_\mathscr{P}(X)$ and conversely.

\section{Prerequisites and Technical terms related to graphs and $C_\mathscr{P}(X)$}

The distance between two distinct vertices $f$ and $g$ in $AG(C_\mathscr{P}(X))$ (respectively, $G(C_\mathscr{P}(X))$), denoted by $d_{AG}(f,g)$ (respectively, $d_{G}(f, g)$), is the length of the shortest path from $f$ to $g$. In what follows, the set of vertices of  any graph $H$ is denoted by $V(H)$. The diameter of a graph $H$ is defined by: $diam(H)=max\{d_H(f,g):f,g\in V(H)\}$. The eccentricity $ecc(f)$ of $f\in V(H)$ is defined by: $ecc(f)=max\{d_H(f,g):g\in V(H)\}$. An $f\in V(H)$ is said to be in the center of $H$ if $ecc(f)\leq ecc(g)$ holds for each $g\in V(H)$ and in this case $ecc(f)$ is called the radius of the graph. The center and radius of a graph $H$ is denoted respectively by $C(H)$ and $rad(H)$. The girth of $H$, denoted by $gr(H)$, is the length of the smallest cycle in this graph. A graph $H$ is called triangulated (respectively, hypertriangulated) if each vertex (respectively, each edge) of this graph is a vertex (respectively, an edge) of a triangle. The smallest length of a cycle containing two distinct vertices $f$ and $g$ in $AG(C_\mathscr{P}(X))$ is denoted by $c(f,g)$.\\
A subset $D$ of $V(H)$ is called a dominating set in $H$ if for each $f\in V(H)\setminus D$, there exists $g\in D$ such that $f$ and $g$ are adjacent. A dominating set $D$ in $H$ is called a total dominating set if for each $f\in V(H)$, there exists $g\in D$ such that $f$ and $g$ are adjacent. The dominating number (respectively, total dominating number) of the graph $H$ is defined as $dt(H) = \min \{|D|:D\text{ is a dominating set in }H\}$ (respectively, $dt_t(H) =\min \{|D|:D\text{ is a total dominating set in }H\}$). It is evident that $dt(H) \leq dt_t(H)$ and for a simple graph $H$, $dt_t(H) \geq 2$. A complete subgraph of a graph $H$ is any subset $K$ of $H$ such that each pair of distinct vertices in $K$ are adjacent. The clique number of $H$ is defined as $cl(H) = \sup \{|K| : K \ is \  a \ complete \ subgraph \ of H\}$. A coloring of a graph is a labeling of the vertices of the graph with colors such that no two adjacent vertices have the same color. More precisely, for a cardinal number $\Lambda$ (finite or infinite), a $\Lambda$-coloring of $H$ is a map $\psi:V(H)\to[0,\Lambda)$ with the following condition: whenever $f,g\in V(H)$ and $f$ and $g$ are adjacent, $\psi(f)\neq\psi(g)$. The chromatic number of $H$ is defined as follows: $\chi(H) =min\{\Lambda:\text{there exists a }\Lambda\text{-coloring of }H\}$.\\
For more information related to graph theoretic terms, the reader is referred to the book \cite{Diestel}. \\

\begin{definition}\cite{Ach}
$X$ is called locally $\mathscr{P}$ at a point $x\in X$, if there exists an open neighborhood $V$ of $x$ in $X$ such that $cl_XV\in\mathscr{P}$. $X$ is said to be locally $\mathscr{P}$ if it is locally $\mathscr{P}$ at each point on it.
\end{definition}
Let $X_\mathscr{P}=\{x\in X: X\text{ is  locally }\mathscr{P}\text{ at }x\}$. Then it is easy to see that $X_\mathscr{P}$ is an open set in $X$. Also, $X$ is locally $\mathscr{P}$ if and only if $X_\mathscr{P}=X$.\\

As a special case of $C_\mathscr{P}(X)$, we get the ring $C(X)$. In fact, if $\mathscr{P}$ is the ideal of all closed subsets of $X$, we get $C_\mathscr{P}(X)=C(X)$ with $X_\mathscr{P}=X$. If $\mathscr{P}$ is the ideal of all compact subsets of $X$ then $C_\mathscr{P}(X)=C_K(X)$ and $X_\mathscr{P}$ is the set of all points at which $X$ is locally compact. On the other hand, if $\mathscr{P}$ is the ideal of all finite subsets of $X$ then $C_\mathscr{P}(X)=C_F(X)=\{f\in C(X): X\setminus Z(f)\text{ is finite}\}$ and $X_\mathscr{P}$ is the set of all isolated points in $X$, denoted later on by $K_X$. 
\begin{lemma}\label{Lem1}\cite{Ach}
Let $x\in X_\mathscr{P}$ and $G$ be a neighborhood of $x$ in $X$. Then there exists $f\in C_\mathscr{P}(X)$ such that $x\in X\setminus Z(f)\subset cl_X(X\setminus Z(f))= X \setminus int_XZ(f)\subset G$.
\end{lemma}
\begin{lemma}\label{Lem2}
Let $K\subset X_\mathscr{P}$ be a compact set and $G$, a neighborhood of $K$ in $X$. Then there exists  $f\in C_\mathscr{P}(X)$ such that $K\subset X\setminus Z(f)\subset cl_X(X\setminus Z(f))\subset G$.
\end{lemma}
\begin{proof}
Let $x\in K$. By Lemma \ref{Lem1}, there exists  $f_x\in C_\mathscr{P}(X)$ such that $x\in X\setminus Z(f_x)\subset cl_X(X\setminus Z(f_x))\subset G$.  $K$ being compact, the open cover $\{X\setminus Z(f_x):x\in K\}$ has a finite subcover; i.e., there exist $x_1, x_2,\ldots, x_n\in K$ such that $K\subset \bigcup\limits_{i=1}^n (X\setminus Z(f_{x_i}))=X\setminus (\bigcap_{i=1}^n Z(f_{x_i}))=X\setminus Z(f)$ where $f=\sum_{i=1}^{n}f_{x_i}^2\in C_\mathscr{P}(X)$. Consequently, $K\subset X\setminus Z(f)\subset cl_X(X\setminus Z(f))=\bigcup\limits_{i=1}^n cl_X(X\setminus Z(f_{x_i}))\subset G$.
\end{proof}
It has been proved in our earlier paper [see Theorem 2.3, \cite{Ach}] that $Z(C_\mathscr{P}(X))^*=\{f\in C_\mathscr{P}(X)\setminus\{0\}:int_XZ(f)\cap X_\mathscr{P}\neq\emptyset\}$.  The following is a couple of easy observations, that we frequently use in our discussion.
\begin{observation}\label{Obs1}
(i) Every point in the co-zero set of $f \in C_\mathscr{P}(X)$ is a $\mathscr{P}$-point.\\
(ii) $fg = 0$ if and only if $X \setminus Z(g) \subseteq int_X Z(f)$, for any $f, g \in C_\mathscr{P}(X)$.
\end{observation}
For more information on the ring $C_\mathscr{P}(X)$, the reader is referred to see the articles \cite{Acharyya} and \cite{Acharyya1}. \\

\section{Annihilator graph of $C_\mathscr{P}(X)$}
We begin this section by introducing the annihilator graph of $C_\mathscr{P}(X)$ whose vertices are the non-zero zero divisors of $C_\mathscr{P}(X)$. The adjacency of two vertices of this graph is reformulated in terms of their zero sets and the $\mathscr{P}$-points of the space $X$.

\begin{definition}\cite{Bad}
The annihilator graph $AG(R)$, of a commutative ring $R$, is a simple graph whose set of vertices is $Z(R)^* = Z(R) \setminus \{0\}$ and two distinct vertices $a,\ b$ are adjacent if $ann(a)\cup ann(b) \neq ann(a.b)$; here $ann(a)$ stands for the ideal $\{r\in R : ra=0\}$
\end{definition}
The following lemma translates a purely algebraic statement into an equivalent topological terms. 
\begin{lemma}\label{Lem4.1}
Let $f,g\in Z(C_\mathscr{P}(X))^*$. Then $int_XZ(f)\cap X_\mathscr{P}\subset int_XZ(g)\cap X_\mathscr{P}$ if and only if $ann(f)\subset ann(g)$.
\end{lemma}
\begin{proof}
Let $int_XZ(f)\cap X_\mathscr{P}\subset int_XZ(g)\cap X_\mathscr{P}$ and $h\in ann(f)$. Then $f.h=0$ and so, by using Observation \ref{Obs1}, $X\setminus Z(h)\subset int_XZ(f)\cap X_\mathscr{P}\subset int_XZ(g)\cap X_\mathscr{P}$, so that $g.h=0$, i.e., $h\in ann(g)$.\\
Conversely let $ann(f)\subset ann(g)$ and $x\in int_XZ(f)\cap X_\mathscr{P}$. Then by Lemma \ref{Lem1}, there exists $h\in C_\mathscr{P}(X)$ such that $x\in X\setminus Z(h)\subset int_XZ(f)\cap X_\mathscr{P}$. Clearly, $h\neq 0$ and $f.h=0$. Now $f.h=0\implies h\in ann(f)\subset ann(g)\implies g.h=0$. Therefore, $x \in X\setminus Z(h)\subset int_XZ(g)\cap X_\mathscr{P}$.  So, $int_XZ(f)\cap X_\mathscr{P}\subset int_XZ(g)\cap X_\mathscr{P}$.
\end{proof}

In light of Lemma \ref{Lem4.1}, the rule for two vertices of $AG(C_\mathscr{P}(X))$ to be adjacent, takes the following form:

\begin{theorem}\label{Th4.2}
Two vertices $f,\ g$ in $AG(C_\mathscr{P}(X))$ are adjacent if and only if $int_XZ(f)\cap X_\mathscr{P} \not\subset int_XZ(g)\cap X_\mathscr{P}$ and $int_XZ(g)\cap X_\mathscr{P}\not\subset int_XZ(f)\cap X_\mathscr{P}$.
\end{theorem}
\begin{proof}
Let $f, \ g$ be adjacent vertices in $AG(C_\mathscr{P}(X))$. Then $ann(f)\cup ann(g)\subsetneqq ann(f.g)$ and so there exists $h\in ann(f.g)\setminus[ann(f)\cup ann(g)]\implies f.g.h=0$ but $f.h\neq 0$ and $g.h\neq 0$. Since $f.h\neq 0$, $(X\setminus Z(f))\cap (X\setminus Z(h))\neq \emptyset$. Choose $x\in (X\setminus Z(f))\cap (X\setminus Z(h))$. By Observation~\ref{Obs1}, $x\in X_\mathscr{P}$. Similarly, $y\in X_\mathscr{P}$ such that $y\in (X\setminus Z(g))\cap (X\setminus Z(h))$. Now, $f.g.h=0\implies (X\setminus Z(f))\cap (X\setminus Z(h)) = X \setminus Z(fh)\subset int_XZ(g)$ and so, $x\in int_XZ(g)\cap X_\mathscr{P}$. Similarly, $y\in int_XZ(f)\cap X_\mathscr{P}$. Therefore, $x\in [int_XZ(g)\cap X_\mathscr{P}]\setminus int_XZ(f)$ and $y\in [int_XZ(f)\cap X_\mathscr{P}]\setminus int_XZ(g)$. So, $int_XZ(f)\cap X_\mathscr{P}\not\subset int_XZ(g)\cap X_\mathscr{P}$ and $int_XZ(g)\cap X_\mathscr{P}\not\subset int_XZ(f)\cap X_\mathscr{P}$.\\
Suppose $f, \ g$ are non-adjacent vertices in $AG(C_\mathscr{P}(X))$. Then by definition, $ann(f)\cup ann(g)=ann(f.g)$. Since the union of two ideals $ann(f)$ and $ann(g)$ is an ideal $ann(f.g)$, it follows that either $ann(f)\subset ann(g)$ or $ann(g)\subset ann(f)$. Then by Lemma~\ref{Lem4.1}, either $int_XZ(f)\cap X_\mathscr{P}\subset int_XZ(g)\cap X_\mathscr{P}$ or $int_XZ(g)\cap X_\mathscr{P}\subset int_XZ(f)\cap X_\mathscr{P}$.
\end{proof}
The following theorem endorses that annihilator graph $AG(C_\mathscr{P}(X))$ is indeed a subgraph of the weakly zero divisor graph $W\Gamma(C_\mathscr{P}(X))$. A weakly zero divisor graph $W\Gamma(R)$ of a commutative ring $R$ is a simple graph with $Z(R)^*$ as the set of vertices and two distinct vertices $a,b$ are adjacent if there exists $c,d\in R$ such that $c\in ann(a)\setminus\{0\}$, $d\in ann(b)\setminus\{0\}$ and $c.d=0$ \cite{Nik}.
\begin{theorem}
$AG(C_\mathscr{P}(X))$ is a subgraph of $W\Gamma(C_\mathscr{P}(X))$.
\end{theorem}
\begin{proof}
Let $f, \ g$ be two adjacent vertices in $AG(C_\mathscr{P}(X))$. By Theorem \ref{Th4.2}, $(int_XZ(f)\cap X_\mathscr{P}) \cap (X \setminus Z(g))\neq \emptyset$ and $(int_XZ(g)\cap X_\mathscr{P}) \cap (X\setminus Z(f))\neq \emptyset$. Therefore, by Lemma \ref{Lem1}, there exist $h_1,\ h_2\in C_\mathscr{P}(X)$ such that $\emptyset\neq X\setminus Z(h_1)\subset cl_X(X\setminus Z(h_1))\subset (int_XZ(f)\cap X_\mathscr{P})\cap (X \setminus Z(g))$ and $\emptyset\neq X\setminus Z(h_2)\subset cl_X(X\setminus Z(h_2))\subset (int_XZ(g)\cap X_\mathscr{P})\cap (X\setminus Z(f))$. So, $h_1,\ h_2\neq 0$ and $f.h_1=0=g.h_2$. i.e., $h_1, \ h_2\in Z(C_\mathscr{P}(X))^*$. Since $(int_XZ(g)\cap X_\mathscr{P})\cap (X \setminus Z(f))\cap (int_XZ(f)\cap X_\mathscr{P})\cap (X\setminus Z(g))=\emptyset$, it follows that $(X\setminus Z(h_1))\cap (X\setminus Z(h_2))=\emptyset $ and so, $h_1.h_2=0$. i.e., $h_1\in ann(f)\setminus\{0\}$, $h_2\in ann(g)\setminus\{0\}$ and $h_1.h_2=0$, proving that $f,\ g$ are adjacent in $W\Gamma(C_\mathscr{P}(X))$. 
\end{proof}
For if $|X_\mathscr{P}|\leq 1$, the ring $C_\mathscr{P}(X)$ is either an integral domain or the trivial ring $0$ and so, $Z(C_\mathscr{P}(X))^*$ is an empty set. Therefore, in case $|X_\mathscr{P}|\leq 1$ , all the three graphs are empty graphs, producing $\Gamma(C_\mathscr{P}(X)) = AG(C_\mathscr{P}(X)) = W\Gamma(C_\mathscr{P}(X))$. Observing that $AG(C_\mathscr{P}(X))$ has the same set of vertices as the other two graphs under consideration and that it lies between $\Gamma(C_\mathscr{P}(X))$ and $W\Gamma(C_\mathscr{P}(X))$, we may refine Theorem 2.8 of \cite{Ach} by allowing the possibility of a graph to become an empty graph, as follows:
\begin{theorem}\label{Eq}
$\Gamma(C_\mathscr{P}(X)) = AG(C_\mathscr{P}(X))=W\Gamma(C_\mathscr{P}(X))$ if and only if $|X_\mathscr{P}|\leq 2$. 
\end{theorem}
Henceforth, in view of Theorem~\ref{Eq}, we carry on our discussion on $AG(C_\mathscr{P}(X))$ with the assumption that the Tychonoff space $X$ is such that $|X_\mathscr{P}| \geq 3$.\\\\
The following theorem provides conditions under which a pair of vertices of $AG(C_\mathscr{P}(X))$ admits of a third vertex adjacent to both of them.  

\begin{theorem}\label{Th4.6}
Let $f,\ g \in Z(C_\mathscr{P}(X))^*$. 
\begin{enumerate}
\item[(i)] If $f$ and $g$ are not adjacent in $\Gamma(C_\mathscr{P}(X))$ then there exists a vertex $h \in Z(C_\mathscr{P}(X))^*$ such that $h$ is adjacent to both $f$ and $g$ in $AG(C_\mathscr{P}(X))$.
\item[(ii)] If $f$ and $g$ are adjacent in $\Gamma(C_\mathscr{P}(X))$, then there exists a vertex $h \in Z(C_\mathscr{P}(X))^*$ such that $h$ is adjacent to both $f$ and $g$ in $AG(C_\mathscr{P}(X))$ if and only if $|int_XZ(f)\cap X_\mathscr{P}|\geq 2$ and $|int_XZ(g)\cap X_\mathscr{P}|\geq 2$.
\end{enumerate}
\end{theorem}
\begin{proof}
(i) Follows from Lemma 2.1(6) and Lemma 2.3 of \cite{Bad}.\\
(ii) If possible let $|int_XZ(f)\cap X_\mathscr{P}|=1$. Since $f$ and $g$ are adjacent in $\Gamma(C_\mathscr{P}(X))$, $fg = 0$ holds. So, $X \setminus Z(g) \subset int_XZ(f) \cap X_\mathscr{P} = \{x\}$ (say). Then $int_XZ(g) = X_\mathscr{P}\setminus\{x\}$ and so, $int_XZ(f)\cap int_XZ(g)\cap X_\mathscr{P} = \emptyset$. Since $h$ is adjacent to $f$ in $AG(C_\mathscr{P}(X))$, by Theorem~\ref{Th4.2}, $\{x\} = int_XZ(f)\cap X_\mathscr{P}$ $ \not\subset  int_XZ(h)\cap X_\mathscr{P}$ $\implies x \notin int_XZ(h)\cap X_\mathscr{P} \implies int_XZ(h)\cap X_\mathscr{P} \subset X_\mathscr{P}\setminus \{x\} = int_XZ(g)\cap X_\mathscr{P}$. By Theorem~\ref{Th4.2} again, $h$ is not adjacent to $g$ in $AG(C_\mathscr{P}(X))$, which is a contradiction.\\
Conversely, let $|int_XZ(f)\cap X_\mathscr{P}|\geq 2$ and $|int_XZ(g)\cap X_\mathscr{P}|\geq 2$. If $int_XZ(f)\cap int_XZ(g)\cap X_\mathscr{P}\neq\emptyset$ then by Lemma \ref{Lem1}, there exists $h\in C_\mathscr{P}(X)$ such that $\emptyset\neq X\setminus Z(h)\subset int_XZ(f)\cap int_XZ(g)\cap X_\mathscr{P}$. So, $h\neq 0$ and $f.h=0=g.h$. i.e., $h\in Z(C_\mathscr{P}(X))^*$ is adjacent to both $f,g$ in $\Gamma(C_\mathscr{P}(X))$ and hence, in $AG(C_\mathscr{P}(X))$. Let $int_XZ(f)\cap int_XZ(g)\cap X_\mathscr{P}=\emptyset$. If we choose $x,y\in int_XZ(f)\cap X_\mathscr{P}$ and $z,w\in int_XZ(g)\cap X_\mathscr{P}$ then $x,y,z,w$ are all distinct. By Lemma \ref{Lem1}, there exists $h_x\in C_\mathscr{P}(X)$ such that $x\in X\setminus Z(h_x)\subset X\setminus int_XZ(h_x)\subset [int_XZ(f)\cap X_\mathscr{P}]\setminus\{y\}\implies \{y\}\cup X\setminus[int_XZ(f)\cap X_\mathscr{P}]\subset int_XZ(h_x)$ and so $y,z,w\in int_XZ(h_x)\cap X_\mathscr{P}$ and $x\in X\setminus Z(h_x)$. Similarly we can find $h_z\in C_\mathscr{P}(X)$ such that $x,y,w\in int_XZ(h_z)\cap X_\mathscr{P}$ and $z\in X\setminus Z(h_z)$. Let $h=h_x^2+h_z^2\in C_\mathscr{P}(X)$ and so, $int_XZ(h)=int_XZ(h_x)\cap int_XZ(h_z)$. Therefore, $y,w\in int_XZ(h)\cap X_\mathscr{P}$ and $x,z\in X\setminus Z(h)$. Clearly, $h\in Z(C_\mathscr{P}(X))^*$. $x\in [int_XZ(f)\cap X_\mathscr{P}]\setminus int_XZ(h)$ and $w\in [int_XZ(h)\cap X_\mathscr{P}]\setminus int_XZ(f)\implies f,h$ are adjacent in $AG(C_\mathscr{P}(X))$, by Theorem \ref{Th4.2}. Similarly, we can show that $g,h$ are adjacent in $AG(C_\mathscr{P}(X))$. Thus $h\in Z(C_\mathscr{P}(X))^*$ is adjacent to both $f,g$ in $AG(C_\mathscr{P}(X))$.
\end{proof}
The next few results directly follow from Theorem \ref{Th4.6}:
\begin{corollary}\label{Cor4.7}
		For $f,g\in Z(C_\mathscr{P}(X))^*$, $$d_{AG}(f,g)=\begin{cases}
			1&\text{ whenever }f,g\text{ are adjacent}\\
			2&\text{ otherwise}
		\end{cases}$$ and hence, $diam(AG(C_\mathscr{P}(X)))=2$.
	\end{corollary}
\begin{corollary}
		For each $f\in Z(C_\mathscr{P}(X))^*$, $ecc(f) = 2$ and hence, $rad(AG(C_\mathscr{P}(X)))=2$ and $C(AG(C_\mathscr{P}(X)))=Z(C_\mathscr{P}(X))^*$.
	\end{corollary}
\begin{proof}
For each $f$, as $2f$ and $f$ are non-adjacent, $ecc(f) =2$ and hence, each $f$ is in the center.
\end{proof}
\begin{corollary}\label{Cor4.9}
		An edge $f-g$ in $AG(C_\mathscr{P}(X))$ is an edge of a triangle in $AG(C_\mathscr{P}(X))$ if and only if either $f-g$ is not an edge in $\Gamma(C_\mathscr{P}(X))$ or else, $|int_XZ(f)\cap X_\mathscr{P}|\geq 2$ and $|int_XZ(g)\cap X_\mathscr{P}|\geq 2$.
\end{corollary}

A topological phenomenon of $X$ produces a sufficient condition for $AG(C_\mathscr{P}(X))$ to be hypertriangulated, as seen in the next theorem.

\begin{theorem}\label{Th4.11}
	$AG(C_\mathscr{P}(X))$ is hypertriangulated if $X_\mathscr{P}$ does not contain any isolated point of $X$.
\end{theorem}
\begin{proof}
	If $X_\mathscr{P}$ does not contain any isolated point then $|int_XZ(f)\cap X_\mathscr{P}|\geq 2$ for every $f\in Z(C_\mathscr{P}(X))^*$. By  Corollary \ref{Cor4.9}, every edge is an edge of a triangle in $AG(C_\mathscr{P}(X))$ and so, $AG(C_\mathscr{P}(X))$ is hypertriangulated.
\end{proof}

The following example establishes that the converse of Theorem \ref{Th4.11} is not always true.
\begin{example}\label{Ex4.13}
Consider the completely regular Hausdorff space $X = \mathbb{N}^*$, the one point compactification of $\mathbb{N}$. If $\mathscr{P}$ is the collection of all finite sets of $X$ then $C_\mathscr{P}(X)=C_F(X)$. Clearly, $X_\mathscr{P} = \mathbb{N}$. If $f\in Z(C_\mathscr{P}(X))^*$ then $X\setminus Z(f)$ is finite and therefore, $Z(f)$ is a clopen set in $X$ such that $X_\mathscr{P} \cap Z(f) = X_\mathscr{P} \cap int_XZ(f)$ is infinite. i.e., $|int_X Z(f) \cap X_\mathscr{P}| \geq 2$, for all $f\in Z(C_\mathscr{P}(X))$. Hence, $AG(C_\mathscr{P}(X))$ is hypertriangulated. 
\end{example}

However, if $X_\mathscr{P}$ contains an isolated point of $X$ and $AG(C_\mathscr{P}(X))$ is hypertriangulated then $X_\mathscr{P}$ is never a clopen subset of $X$. For if  $X_\mathscr{P}$ is clopen and $x\in X_\mathscr{P}$ is an isolated point of $X$, taking $f=1_x$ and $g=1_{X_\mathscr{P}\setminus\{x\}}$ we get $f.g=0$. Since both $X_\mathscr{P}$ and $\{x\}$ are clopen sets in $X$, $f,g\in C_\mathscr{P}(X)$ and therefore, $f$ and $g$ are adjacent in $AG(C_\mathscr{P}(X))$. Since $int_XZ(g)\cap X_\mathscr{P}=\{x\}$, by Theorem $\ref{Th4.6}$(ii), no vertex is adjacent to both $f$ and $g$ in $AG(C_\mathscr{P}(X))$. Consequently, $f-g$ in $AG(C_\mathscr{P}(X))$ is not an edge of a triangle, i.e., $AG(C_\mathscr{P}(X))$ is not hypertriangulated.

We record the results related to hypertriangulatedness of $AG(C_\mathscr{P}(X))$ for special choices of $\mathscr{P}$.
\begin{theorem}
\begin{enumerate}
\item $AG(C(X))$ is hypertriangulated if and only if $X$ has no isolated point.
\item $AG(C_K(X))$ is hypertriangulated if and only if $X$ has no isolated point.
\item $AG(C_F(X))$ is hypertriangulated if and only if the number of isolated points of $X$ is infinite. 
\end{enumerate}  
\end{theorem}
\begin{proof}
\begin{enumerate}
\item If $\mathscr{P}$ is the ideal of all closed subsets of $X$ then $C_\mathscr{P}(X)=C(X)$ and $X_\mathscr{P}=X$ (= clopen in $X$). So, $AG(C(X))$ is hypertriangulated if and only if $X$ has no isolated point.
\item If $\mathscr{P}$ is the ideal of all compact subsets of $X$ then $C_\mathscr{P}(X)=C_K(X)$ and $X_\mathscr{P}$ is the set of all points at which $X$ is locally compact. Therefore, for a locally compact space $X$, $AG(C_K(X))$ is hypertriangulated if and only if $X$ has no isolated point.
\item If $\mathscr{P}$ is the ideal of all finite subsets of $X$ then $C_\mathscr{P}(X)=C_F(X)=\{f\in C(X): X\setminus Z(f)\text{ is finite}\}$ and $X_\mathscr{P}$ is the set of all isolated points in $X$. It is to be noted that $C_\mathscr{P}(X)$ is a non-zero ring if and only if $X$ contains atleast one isolated point. So, if the number of isolated points of $X$ is finite then certainly $X_\mathscr{P}$ is a clopen subset of $X$ and hence, $AG(C_F(X))$ is never hypertriangulated. Let the number of isolated points of $X$ be infinite. Then proceeding as in Example~\ref{Ex4.13}, $AG(C_F(X))$ is hypertriangulated. 
\end{enumerate}  
\end{proof}

\begin{theorem}\label{Th4.12}
		$AG(C_\mathscr{P}(X))$ is always triangulated.
	\end{theorem}
	\begin{proof}
		Let $f\in Z(C_\mathscr{P}(X))^*$. Then $int_XZ(f)\cap X_\mathscr{P}\neq\emptyset$. If $|int_XZ(f)\cap X_\mathscr{P}|=1$, say $int_XZ(f)\cap X_\mathscr{P}=\{x\}$, then as $|X_\mathscr{P}|\geq 3$, we can find distinct $y,z$ in $X_\mathscr{P}\setminus\{x\}$. By Lemma~\ref{Lem2}, there exists $g,h\in C_\mathscr{P}(X)$ such that $\{x,y\}\subset X\setminus Z(g)\subset X\setminus int_XZ(g)\subset X_\mathscr{P}\setminus\{z\}$ and $\{x,z\}\subset X\setminus Z(h)\subset X\setminus int_XZ(h)\subset X_\mathscr{P}\setminus\{y\}$. Then $g,h\neq 0$, $z\in int_XZ(g)\cap X_\mathscr{P}$ and $y\in int_XZ(h)\cap X_\mathscr{P}\implies g,h\in Z(C_\mathscr{P}(X))^*$. Now $x\in [int_XZ(f)\cap X_\mathscr{P}]\setminus int_XZ(g)$ and $y\in [int_XZ(g)\cap X_\mathscr{P}]\setminus int_XZ(f)\implies f,g$ are adjacent in $AG(C_\mathscr{P}(X))$. Similarly, $g,h$ and $f,g$ are adjacent in $AG(C_\mathscr{P}(X))$. Thus $f-g-h-f$ constitutes a triangle with $f$ as a vertex in $AG(C_\mathscr{P}(X))$.\\
		If $|int_XZ(f)\cap X_\mathscr{P}|\geq 2$ then by Theorem 2.22 of \cite{Ach}, $f$ is a vertex of a triangle in $\Gamma_\mathscr{P}(X)$ and hence, a vertex of a triangle in $AG(C_\mathscr{P}(X))$. 
	\end{proof}

Before calculating eccentricity, girth, radius and length of shortest cycle containing two specific vertices, we jot down two very elementary facts which are amply clear and need no further clarification:\\
In $AG(C_\mathscr{P}(X))$, 
\begin{enumerate} 
\item[(i)]  $f$ and $g$ are adjacent if and only if $m_1f$ and $m_2 g$ are adjacent, for all positive integers $m_1, m_2 $. 
\item[(ii)] $m_1f$ and $m_2 f$, for any positive integers $m_1, m_2$, are not adjacent.
\end{enumerate}
In view of these elementary facts, the following are immediate:
\begin{theorem}\label{Th4.15}
In $AG(C_\mathscr{P}(X))$,
\begin{enumerate}
\item[(i)] Any edge $f-g$ is an edge of a square and so, $3 \leq c(f, g) \leq 4$. Furthermore,
for any $f,g\in Z(C_\mathscr{P}(X))^*$, 
		\begin{enumerate}
			\item[(a)] if $f,g$ are adjacent in $AG(C_\mathscr{P}(X))$ with $f.g\neq 0$, then $c(f,g)=3$;
			\item[(b)] if $f,g$ are adjacent in $AG(C_\mathscr{P}(X))$ with $f.g= 0$ then \\ $c(f,g)=\begin{cases}
				3&\text{ if }|int_XZ(f)\cap X_\mathscr{P}|\geq 2\text{ and } |int_XZ(g)\cap X_\mathscr{P}|\geq 2\\
				4&\text{ if }|int_XZ(f)\cap X_\mathscr{P}|=1\text{ or } |int_XZ(g)\cap X_\mathscr{P}|=1
			\end{cases}$
			\item[(c)] if $f,g$ are non-adjacent in $AG(C_\mathscr{P}(X))$ then $c(f,g)=4$.
		\end{enumerate}
\item[(ii)] $gr(AG(C_\mathscr{P}(X))) = 3$.
 
\end{enumerate}

\end{theorem}
\begin{proof}
(i) Since $f-g$ is an edge in $AG(C_\mathscr{P}(X))$, it follows that $f-g-2f-2g-f$ completes a square. Consequently, $3 \leq c(f, g) \leq 4$. 
\begin{enumerate}
			\item[(a)] $f,g$ are adjacent in $AG(C_\mathscr{P}(X))$ with $f.g\neq 0$ means that $f,g$ are not adjacent in $\Gamma(C_\mathscr{P}(X))$ and hence, by Theorem~\ref{Th4.6}(i), there exists a vertex $h \in Z(C_\mathscr{P}(X))^*$ such that $f-h-g-f$ constitutes a cycle of length 3 in $AG(C_\mathscr{P}(X))$. i.e., $c(f, g) = 3$.
			\item[(b)] By Theorem~\ref{Th4.6}(ii) there exists a vertex $h \in Z(C_\mathscr{P}(X))^*$ such that $f-h-g-f$ constitutes a cycle of length 3 in $AG(C_\mathscr{P}(X))$ if and only if $|int_X Z(f)\cap X_\mathscr{P}|\geq 2$ and  $|int_X Z(g)\cap X_\mathscr{P}|\geq 2$. So, in this case, \\
$c(f,g)=\begin{cases}
				3 &\text{ if }|int_XZ(f)\cap X_\mathscr{P}|\geq 2\text{ and } |int_XZ(g)\cap X_\mathscr{P}|\geq 2\\
				4 &\text{ if }|int_XZ(f)\cap X_\mathscr{P}|=1\text{ or } |int_XZ(g)\cap X_\mathscr{P}|=1
			\end{cases}$
			\item[(c)] If $f,g$ are non-adjacent in $AG(C_\mathscr{P}(X))$ then by Theorem~\ref{Th4.6}(i), there exists a vertex $h \in Z(C_\mathscr{P}(X))^*$ such that $h$ is adjacent to both of $f$ and $g$, constituting a square $f-h-g-2h-f$. Therefore, $c(f, g) = 4$. 
		\end{enumerate}
(ii) Follows from the fact that $AG(C_\mathscr{P}(X))$ is triangulated.
\end{proof}
\begin{corollary}
Each chord-less cycle in $AG(C_\mathscr{P}(X))$ is of length $3$ or $4$.
\end{corollary}
The crux of Theorem~\ref{Th4.15} may be represented diagrammatically as follows:
\begin{center}
		\begin{tikzpicture}
			\draw (0,0)--(2,2)--(4,0)--(0,0) (7,0)--(9,2)--(11,0)--(7,0);
			\coordinate[Bullet=black,label=left:$f$] (f) at (0,0);
			\coordinate[Bullet=black,label=right:$g$] (g) at (4,0);
			\coordinate[Bullet=black,label=left:$f$] (f) at (7,0);
			\coordinate[Bullet=black,label=right:$g$] (g) at (11,0);
			\node (start) at (2,-1) {$f.g\neq 0\text{ and }f,g\text{ are adjacent}$};
			\node (start) at (9,-1) {$f.g=0\text{ and }|int_XZ(f)\cap X_\mathscr{P}|\geq 2$};
			\node (start) at (9,-1.5) {and $|int_XZ(g)\cap X_\mathscr{P}|\geq 2$};
			\draw (0,-5)--(0,-3)--(4,-3)--(4,-5)--(0,-5) (7,-4)--(9,-3)--(11,-4)--(9,-5)--(7,-4);
			\coordinate[Bullet=black,label=left:$f$] (f) at (0,-5);
			\coordinate[Bullet=black,label=right:$g$] (g) at (4,-5);
			\coordinate[Bullet=black,label=left:$f$] (f) at (7,-4);
			\coordinate[Bullet=black,label=right:$g$] (g) at (11,-4);
			\node (start) at (2,-6) {$f.g=0\text{ and }|int_XZ(f)\cap X_\mathscr{P}|=1$};
			\node (start) at (9,-6) {$f,g\text{ are non-adjacent}$};
			\node (start) at (2,-6.5) {or $|int_XZ(g)\cap X_\mathscr{P}|=1$};
		\end{tikzpicture}
	\end{center}
The following theorem refines Theorem~\ref{Th4.15} in absence of any isolated point of $X$ in $X_\mathscr{P}$. 
\begin{theorem}\label{Th4.19}
		If $X_\mathscr{P}$ contains no isolated point of $X$, then for any $f,g\in Z(C_\mathscr{P}(X))^*$, $$c(f,g)=\begin{cases}
			3&\text{ if }f,g\text{ are adjacent}\\
			4&\text{ otherwise}
		\end{cases}$$ 
	\end{theorem}
	\begin{proof}
		If $X_\mathscr{P}$ contains no isolated point of $X$ then $|int_XZ(f)\cap X_\mathscr{P}|\geq 2$ for all $f\in Z(C_\mathscr{P}(X))^*$. So, by Theorem \ref{Th4.15}, the result holds.
	\end{proof}
Example \ref{Ex4.13} guarantees that the converse of Theorem~\ref{Th4.19} doesn't always hold. More generally, in any completely regular Hausdorff space $X$ with infinitely many isolated points, as $X \setminus Z(f)$, for each $f \in (Z(C_F(X)))^*$, is a finite subset of isolated points, $int_XZ(f)\cap X_\mathscr{P}= Z(f) \cap X_\mathscr{P}$ is infinite. As a consequence, $|int_XZ(f)\cap X_\mathscr{P}|\geq 2$ holds for all $f \in (Z(C_F(X)))^*$. So, by Theorem~\ref{Th4.15}, in spite of the presence of infinitely many isolated points of $X$, for every pair $f, g \in (Z(C_F(X)))^*$,
$$c(f,g)=\begin{cases}
			3&\text{ if }f,g\text{ are adjacent}\\
			4&\text{ otherwise}
		\end{cases}$$ 
In the next theorem we show that a sufficient condition for the converse of Theorem~\ref{Th4.19} to hold is that $X_\mathscr{P}$ is a clopen (not necessarily proper) subset of $X$. 
\begin{theorem}\label{Th4.19a}
If $X_\mathscr{P}$ is a clopen (not necessarily proper) subset of $X$ and for any $f,g\in Z(C_\mathscr{P}(X))^*$, $$c(f,g)=\begin{cases}
			3&\text{ if }f,g\text{ are adjacent}\\
			4&\text{ otherwise}
		\end{cases}$$ 
then $X_\mathscr{P}$ does not contain any isolated point of $X$. 
\end{theorem}
\begin{proof}
let $X_\mathscr{P}$ be a clopen subset of $X$. If $x\in X_\mathscr{P}$ is an isolated point of $X$ then choosing $f=1_x\in Z(C_\mathscr{P}(X))^*$ and $g=1_{X_\mathscr{P}\setminus \{x\}}\in Z(C_\mathscr{P}(X))^*$ we find that $f.g=0$ and $int_XZ(g)\cap X_\mathscr{P}=\{x\}$. Then by Theorem \ref{Th4.15}, $c(f,g)=4$ though $f,g$ are adjacent in $AG(C_\mathscr{P}(X))$, contradicting the hypothesis.
\end{proof}

Interpreting Theorem~\ref{Th4.19}and Theorem~\ref{Th4.19a} for special choices of $\mathscr{P}$, we get an explicit formula for $c(f, g)$ in terms of the adjacency of the vertices $f$ and $g$ in $C(X)$ as well as in $C_K(X)$. 

\begin{theorem} Let $X$ be a completely regular Hausdorff space. 
\begin{enumerate}
\item $X$ has no isolated point if and only if for any $f,g\in Z(C(X))^*$,  
$$c(f,g)=\begin{cases}
					3&\text{ if }f,g\text{ are adjacent}\\
					4&\text{ otherwise}
				\end{cases}$$
\item If $X$ is locally compact then $X$ has no isolated point if and only if for any $f,g\in Z(C_K(X))^*$,  
$$c(f,g)=\begin{cases}
					3&\text{ if }f,g\text{ are adjacent}\\
					4&\text{ otherwise}.
				\end{cases}$$
			\end{enumerate}
\end{theorem}
Two vertices $a, b$ in a graph $G$ are called orthogonal, denoted by $a\perp b$, if $a, b$ are adjacent in $G$ and no other vertex of $G$ is adjacent to both of $a$ and $b$. The following theorem sets a criterion for existence of a pair of orthogonal vertices in $AG(C_\mathscr{P}(X))$ in terms of the topological behaviour of $X$ and $X_\mathscr{P}$.
\begin{theorem}\label{Th4.18}
		For any $f,g\in Z(C_\mathscr{P}(X))^*$, $f\perp g$ in $AG(C_\mathscr{P}(X))$ if and only if $f.g=0$ and either $|int_XZ(f)\cap X_\mathscr{P}|=1$ or $|int_XZ(g)\cap X_\mathscr{P}|=1$.
	\end{theorem}
	\begin{proof}
		If $f\perp g$  then $f,g$ are adjacent in $AG(C_\mathscr{P}(X))$ and no other vertex is adjacent to both $f,g$ in $AG(C_\mathscr{P}(X))$. By Theorem~\ref{Th4.6}(i), $f.g\neq 0$ is impossible. So, $f.g=0$ and therefore by Theorem~\ref{Th4.6}(ii), either $|int_XZ(f)\cap X_\mathscr{P}|=1$ or $|int_XZ(g)\cap X_\mathscr{P}|=1$.\\
		Converse follows from Theorem~\ref{Th4.6}(ii) and the definition of $f\perp g$.
	\end{proof}

\begin{theorem}\label{Th4.22}
		$AG(C_\mathscr{P}(X))$ does not contain any orthogonal pair of vertices if $X_\mathscr{P}$ has no isolated point of $X$. The converse holds if $X_\mathscr{P}$ is a clopen (not necessarily proper) subset of $X$.
	\end{theorem}
	\begin{proof}
	$AG(C_\mathscr{P}(X))$ contains an orthogonal pair of vertices $f, g$ then by Theorem~\ref{Th4.18}, either $|int_XZ(f)\cap X_\mathscr{P}|$ or $|int_XZ(g)\cap X_\mathscr{P}|$ is a singleton, say $\{x\}$. Since $X_\mathscr{P}$ is open in $X$, it follows that $x \in X_\mathscr{P}$ is an isolated point of $X$, contradicting our hypothesis. \\
	Conversely let $X_\mathscr{P}$ be a clopen subset of $X$ and $x\in X_\mathscr{P}$ be an isolated point of $X$. By choosing $f=1_x\in Z(C_\mathscr{P}(X))^*$ and $g=1_{X_\mathscr{P}\setminus\{x\}}\in Z(C_\mathscr{P}(X))^*$, it follows from Theorem~\ref{Th4.18} that $f\perp g$ in $AG(C_\mathscr{P}(X))$. 
	\end{proof}
Once again, Example~\ref{Ex4.13} endorses that abundance of isolated points in $X_\mathscr{P}$ does not ensure the existence of a pair of orthogonal vertices in $AG(C_\mathscr{P}(X))$.\\

A graph $G$ is said to be complemented if for each vertex $a$ in $G$, there exists a vertex $b$ in $G$ such that $a \perp b$. A complemented graph $G$ is called uniquely complemented if $a \perp b$ and $a\perp c$ implies that $b$ and $c$ are adjacent to the same set of vertices in $G$. It is well known that a graph may be complemented without being uniquely complemented. In the next theorem we prove that for $AG(C_\mathscr{P}(X))$, the terms `complemented' and `uniquely complemented' are synonymous. 
\begin{theorem}\label{Th4.20}
		If $AG(C_\mathscr{P}(X))$ is complemented then it is uniquely complemented.
\end{theorem}
	\begin{proof}
		Let $f\perp g$ and $f\perp h$ in $AG(C_\mathscr{P}(X))$ for $f,g,h\in Z(C_\mathscr{P}(X))^*$. By Theorem \ref{Th4.18}, $f.g=0=f.h$ and either $|int_XZ(f)\cap X_\mathscr{P}|=1$ or $|int_XZ(g)\cap X_\mathscr{P}|=1=|int_XZ(h)\cap X_\mathscr{P}|$.\\
	If $int_XZ(f)\cap X_\mathscr{P}=\{x\}$ then $f.g=0\implies X\setminus Z(g)\subset int_XZ(f)\cap X_\mathscr{P}=\{x\}$, so that, $int_XZ(g)\cap X_\mathscr{P}=X_\mathscr{P}\setminus\{x\}$. Similarly, $f.h=0\implies int_XZ(h)\cap X_\mathscr{P}=X_\mathscr{P}\setminus\{x\}$. Hence, $int_XZ(g)\cap X_\mathscr{P}=int_XZ(h)\cap X_\mathscr{P}$.\\
	If $|int_XZ(g)\cap X_\mathscr{P}|=1=|int_XZ(h)\cap X_\mathscr{P}|$ then assuming $int_XZ(g)\cap X_\mathscr{P}=\{x\}$ and $int_XZ(h)\cap X_\mathscr{P}=\{y\}$, we arrive at the following : $f.g=0\implies X\setminus Z(f)\subset int_XZ(g)\cap X_\mathscr{P}=\{x\}$ and $f.h=0\implies X\setminus Z(f)\subset int_XZ(h)\cap X_\mathscr{P}=\{y\}$. Therefore, $X \setminus Z(f) = x=y$ and hence, $int_XZ(g)\cap X_\mathscr{P}=int_XZ(h)\cap X_\mathscr{P}$.\\
		So, in any case, $int_XZ(g)\cap X_\mathscr{P}=int_XZ(h)\cap X_\mathscr{P}$. By Theorem \ref{Th4.2}, it is evident that $g,h$ are adjacent to the same set of vertices in $AG(C_\mathscr{P}(X))$. 
	\end{proof}

If $|X_\mathscr{P}|=2$ then $AG(C_\mathscr{P}(X)) (= \Gamma(C_\mathscr{P}(X)))$ becomes a complete bipartite graph and hence, it is uniquely complemented. For if $|X_\mathscr{P}|<2$, $AG(C_\mathscr{P}(X))$ remains (trivially) uniquely complemented as well. So, being confined to our blanket assumption that $|X_\mathscr{P}|\geq 3$, we prove next that $AG(C_\mathscr{P}(X))$ is almost never a (uniquely) complemented graph. \\
\begin{theorem}
	$AG(C_\mathscr{P}(X))$ is uniquely complemented if and only if $|X_\mathscr{P}|= 3$.
\end{theorem}
	\begin{proof}
		In view of Theorem~\ref{Th4.20}, it is enough to show that the graph is complemented if and only if $|X_\mathscr{P}|= 3$. \\
Let $X_\mathscr{P}=\{x,y,z\}$. So, $f\in Z(C_\mathscr{P}(X))^*$ implies that either $|X\setminus Z(f)|=1$ or $2$.\\
	If $|X\setminus Z(f)|=1$, say $X\setminus Z(f)=\{x\}$, then $g=1_{\{y,z\}} \in Z(C_\mathscr{P}(X))^*$ such that $f.g=0$ and $|int_XZ(g)\cap X_\mathscr{P}|=1$. By Theorem \ref{Th4.18}, $f\perp g$ in $AG(C_\mathscr{P}(X))$.\\
	If $|X\setminus Z(f)|=2$, say $X\setminus Z(f)=\{x,y\}$, then $|int_XZ(f)\cap X_\mathscr{P}|=1$. By choosing $g=1_z$ it is easy to see that $g \in Z(C_\mathscr{P}(X))^*$ with $f.g=0$. By Theorem \ref{Th4.18} again, $f\perp g$ in $AG(C_\mathscr{P}(X))$. \\
	For if $|X_\mathscr{P}|\geq 4$ there are atleast four distinct vertices $x, y, z, w \in X_\mathscr{P}$. By Lemma \ref{Lem2}, there exist $f_x,f_y\in Z(C_\mathscr{P}(X))^*$ such that $\{y,z,w\}\subset X\setminus Z(f_x)\subset X\setminus int_XZ(f_x)\subset X\setminus\{x\}$ and $\{x,z,w\}\subset X\setminus Z(f_y)\subset X\setminus int_XZ(f_y)\subset X\setminus\{y\}$. Then $x\in int_XZ(f_x)\cap X_\mathscr{P}$ and $y\in int_XZ(f_y)\cap X_\mathscr{P}$. Let $f=f_x.f_y\in C_\mathscr{P}(X)$. Then $x,y\in int_XZ(f)\cap X_\mathscr{P}\implies f\in Z(C_\mathscr{P}(X))^*$ and $z,w\in X\setminus Z(f)$. If possible let $g\in Z(C_\mathscr{P}(X))^*$ be such that $f\perp g$ in $AG(C_\mathscr{P}(X))$. By Theorem \ref{Th4.18}, $f.g=0$ and $|int_XZ(g)\cap X_\mathscr{P}|=1$ (as $|int_XZ(f)\cap X_\mathscr{P}|\geq 2$). But $f.g=0\implies X\setminus Z(f)\subset int_XZ(g)\cap X_\mathscr{P}\implies z,w\in int_XZ(g)\cap X_\mathscr{P}$ i.e., $|int_XZ(g)\cap X_\mathscr{P}|\geq 2$ and so, we arrive at a contradiction. Hence, no vertex is orthogonal to such $f$ in $AG(C_\mathscr{P}(X))$. i.e., $AG(C_\mathscr{P}(X))$ is not complemented when $|X_\mathscr{P}|\geq 4$.
	\end{proof}
In particular, special choices of $\mathscr{P}$ reveals when $C(X)$ and $C_K(X)$ are (uniquely) complemented, as recorded in the following Corollary:
\begin{corollary}
\begin{enumerate}
\item $AG(C(X))$ is uniquely complemented if and only if $X$ has atmost three points.
\item $AG(C_K(X))$ is uniquely complemented if and only if $X$ is locally compact at atmost three points.
\end{enumerate}
\end{corollary}

\begin{lemma}\label{Lem4.22}
		Let $f,g\in Z(C_\mathscr{P}(X))^*$ such that $f.g=0$ and $int_XZ(f)\cap int_XZ(g)\cap X_\mathscr{P}=\emptyset$. Then  for some $h\in C_\mathscr{P}(X)$, $[int_XZ(f)\cup int_XZ(g)]\cap X_\mathscr{P}\subset int_XZ(h)\cap X_\mathscr{P}$ implies $h=0$.
	\end{lemma}
	\begin{proof}
		$f.g=0\implies X\setminus Z(f)\subset int_XZ(g)\cap X_\mathscr{P}$ and $X\setminus Z(g)\subset int_XZ(f)\cap X_\mathscr{P}$. Therefore, $X\setminus Z(f)\cup X\setminus Z(g)\subset [int_XZ(g)\cap X_\mathscr{P}]\cup[int_XZ(f)\cap X_\mathscr{P}]=[int_XZ(f)\cup int_XZ(g)]\cap X_\mathscr{P}\subset int_XZ(h)\cap X_\mathscr{P}\implies X\setminus[int_XZ(h)\cap X_\mathscr{P}]\subset Z(f)\cap Z(g)$. Thus $X\setminus Z(h)\subset X\setminus[int_XZ(h)\cap X_\mathscr{P}]\subset Z(f)\cap Z(g)\cap X_\mathscr{P}\implies X\setminus Z(h)\subset int_X[Z(f)\cap Z(g)\cap X_\mathscr{P}]= int_XZ(f)\cap int_XZ(g)\cap X_\mathscr{P}=\emptyset\implies Z(h)=X$, i.e., $h=0$. 
	\end{proof}
	\begin{theorem}\label{Th4.23}
		$dt(AG(C_\mathscr{P}(X)))=2$ if and only if there exists a pair of vertices $f,g$ in $AG(C_\mathscr{P}(X))$ such that $f.g=0$ and $int_XZ(f)\cap int_XZ(g)\cap X_\mathscr{P}=\emptyset$.
	\end{theorem}
	\begin{proof}
		Let $dt(AG(C_\mathscr{P}(X)))=2$ and $\{f,g\}$ be a dominating set in $AG(C_\mathscr{P}(X))$. Since $2f$ is not adjacent to $f$ and $2f\notin \{f,g\}$, $2f$ is adjacent to $g$ and so $f$ is adjacent to $g$ in $AG(C_\mathscr{P}(X))$. By Theorem \ref{Th4.2}, we get
		\begin{equation}\label{eq1}
			\begin{split}
					int_XZ(f)\cap X_\mathscr{P}\not\subset int_XZ(g)\cap X_\mathscr{P}\text{ and }\\
					int_XZ(g)\cap X_\mathscr{P}\not\subset int_XZ(f)\cap X_\mathscr{P}
			\end{split}
		\end{equation}
	 	If possible let $f.g\neq 0$ and $x\in X\setminus Z(f)\cap X\setminus Z(g)$. Then by Lemma~\ref{Lem1}, there exists $h\in C_\mathscr{P}(X)$ such that $x\in X\setminus Z(h)\subset X\setminus int_XZ(h)\subset X\setminus Z(f)\cap X\setminus Z(g)\implies h\neq 0$ and $Z(f)\cup Z(g)\subset int_XZ(h)\implies h\in Z(C_\mathscr{P}(X))^*$ and $h$ is not adjacent to both $f,g$ in $AG(C_\mathscr{P}(X))$. By Equation \ref{eq1} and the fact that $Z(f)\cup Z(g)\subset int_XZ(h)$, we get $h\neq f$ and $h\neq g$, i.e., $h\notin \{f,g\}$ which contradicts that $\{f,g\}$ is a dominating set in $AG(C_\mathscr{P}(X))$. \\
		If possible let $int_XZ(f)\cap int_XZ(g)\cap X_\mathscr{P}\neq\emptyset$. Choosing $h=f^2+g^2\in C_\mathscr{P}(X)$, $int_XZ(h)=int_XZ(f)\cap int_XZ(g)\implies int_XZ(h)\cap X_\mathscr{P}\neq\emptyset\implies h\in Z(C_\mathscr{P}(X))^*$ and $h\neq f$, $h\neq g$ by Equation \ref{eq1}. Since $int_XZ(f)\cap int_XZ(g)\cap X_\mathscr{P}=int_XZ(h)\cap X_\mathscr{P}$, $h\notin\{f,g\}$ is not adjacent to both $f,g$ in $AG(C_\mathscr{P}(X))$, which gives a contradiction. \\
		Conversely, let $f,g\in Z(C_\mathscr{P}(X))^*$ be such that $f.g=0$ and $int_XZ(f)\cap int_XZ(g)\cap X_\mathscr{P}=\emptyset$. Since for each $h\in Z(C_\mathscr{P}(X))^*$, $ecc(h) =2$ implies that $dt(AG(C_\mathscr{P}(X)))\geq 2$. We now show that $\{f,g\}$ is a dominating set in $AG(C_\mathscr{P}(X))$. If possible let $h$ be not adjacent to both $f$ and $g$ in $AG(C_\mathscr{P}(X))$. Then by Theorem~\ref{Th4.2}, [either $int_XZ(f)\cap X_\mathscr{P}\subset int_XZ(h)\cap X_\mathscr{P}$ or $int_XZ(h)\cap X_\mathscr{P}\subset int_XZ(f)\cap X_\mathscr{P}$] and [either $int_XZ(g)\cap X_\mathscr{P}\subset int_XZ(h)\cap X_\mathscr{P}$ or $int_XZ(h)\cap X_\mathscr{P}\subset int_XZ(g)\cap X_\mathscr{P}$].\\
		If $int_XZ(f)\cap X_\mathscr{P}\subset int_XZ(h)\cap X_\mathscr{P}$ and $int_XZ(g)\cap X_\mathscr{P}\subset int_XZ(h)\cap X_\mathscr{P}$ then $[int_XZ(f)\cup int_XZ(g)]\cap X_\mathscr{P}\subset int_XZ(h)\cap X_\mathscr{P}$ and hence by Lemma \ref{Lem4.22}, we get $h=0$ which is not possible as $h\in Z(C_\mathscr{P}(X))^*$.\\
		If $int_XZ(h)\cap X_\mathscr{P}\subset int_XZ(f)\cap X_\mathscr{P}$ and $int_XZ(h)\cap X_\mathscr{P}\subset int_XZ(g)\cap X_\mathscr{P}$ then $int_XZ(h)\cap X_\mathscr{P}\subset int_XZ(f)\cap int_XZ(g)\cap X_\mathscr{P}=\emptyset$ which contradicts that $h\in Z(C_\mathscr{P}(X))^*$.\\
		If $int_XZ(h)\cap X_\mathscr{P}\subset int_XZ(f)\cap X_\mathscr{P}$ and $int_XZ(g)\cap X_\mathscr{P}\subset int_XZ(h)\cap X_\mathscr{P}$ then $int_XZ(g)\cap X_\mathscr{P}\subset int_XZ(f)\cap X_\mathscr{P}\implies int_XZ(f)\cap int_XZ(g)\cap X_\mathscr{P}= int_XZ(g)\cap X_\mathscr{P}$. Therefore $int_XZ(g)\cap X_\mathscr{P}=\emptyset$ which contradicts that $g\in Z(C_\mathscr{P}(X))^*$.\\
		If $int_XZ(f)\cap X_\mathscr{P}\subset int_XZ(h)\cap X_\mathscr{P}$ and $int_XZ(h)\cap X_\mathscr{P}\subset int_XZ(g)\cap X_\mathscr{P}$ then $int_XZ(f)\cap X_\mathscr{P}\subset int_XZ(g)\cap X_\mathscr{P}$ and we get a similar contradiction.\\
		Hence, $\{f,g\}$ is a dominating set in $AG(C_\mathscr{P}(X))$, proving $dt(AG(C_\mathscr{P}(X)))=2$.
	\end{proof}
	\begin{theorem}\label{Th4.26}
		$dt(AG(C_\mathscr{P}(X)))=2$ if and only if $\Gamma(C_\mathscr{P}(X))$ is not hypertriangulated.
	\end{theorem}
	\begin{proof}
		By Theorem $2.28$ in \cite{Ach}, $\Gamma(C_\mathscr{P}(X))$ is hypertriangulated if and only if for any edge $f-g$ in $\Gamma(C_\mathscr{P}(X))$, $int_XZ(f)\cap int_XZ(g)\cap X_\mathscr{P}\neq\emptyset$. Let $\Gamma(C_\mathscr{P}(X))$ be not hypertriangulated. Then there is an edge $f-g$ in $\Gamma(C_\mathscr{P}(X))$ such that $int_XZ(f)\cap int_XZ(g)\cap X_\mathscr{P}=\emptyset$. Since $f-g$ is an edge in $\Gamma(C_\mathscr{P}(X))$, $f.g=0$. Hence, by Theorem \ref{Th4.23}, $dt(AG(C_\mathscr{P}(X)))=2$. Conversely, if $dt(AG(C_\mathscr{P}(X))) =2$ then there exists $f,g\in Z(C_\mathscr{P}(X))^*$ such that $f.g=0$ and $int_XZ(f)\cap int_XZ(g)\cap X_\mathscr{P}=\emptyset$. Then $f.g=0\implies f-g$ is an edge in $\Gamma(C_\mathscr{P}(X))$ and $int_XZ(f)\cap int_XZ(g)\cap X_\mathscr{P}=\emptyset\implies f-g$ is not an edge of a triangle in $\Gamma(C_\mathscr{P}(X))$. Therefore, $\Gamma(C_\mathscr{P}(X))$ is not hypertriangulated.
	\end{proof}
	The following Corollary is an immediate consequence of Theorem \ref{Th4.26} and Theorem $2.29$ in \cite{Ach}.
	\begin{corollary}
		If $cl_X(X_\mathscr{P})\notin \mathscr{P}$, then $dt(AG(C_\mathscr{P}(X)))>2$. 
	\end{corollary}
	\begin{theorem}\label{Th4.25}
		If $X_\mathscr{P}$ is clopen and properly contains a clopen subset of $X$ then $$dt(AG(C_\mathscr{P}(X)))=2.$$
	\end{theorem}
	\begin{proof}
		Since $X_\mathscr{P}$ is a disconnected subspace of $X$,  $X_\mathscr{P}$ contains a non-empty proper clopen subset, say $K$. Since $X_\mathscr{P}$ is clopen in $X$, $K$ is clopen in $X$ also. Consider $f=1_K$ and $g=1_{X_\mathscr{P}\setminus K}$. Since both $K$ and $X_\mathscr{P}\setminus K$ are clopen in $X$, $f,g\in C_\mathscr{P}(X)$ and $f.g=0$. Also $int_XZ(f)\cap X_\mathscr{P}= X_\mathscr{P}\setminus K$ and $int_XZ(f)\cap X_\mathscr{P}=K\implies int_XZ(f)\cap int_XZ(g)\cap X_\mathscr{P}=\emptyset$ and $f,g\in Z(C_\mathscr{P}(X))^*$. So, by Theorem~\ref{Th4.23}, $dt(AG(C_\mathscr{P}(X)))=2$.
	\end{proof}
	The converse of Theorem \ref{Th4.25} is not true in general. For example let $X=\mathbb{R}$ and $\mathscr{P}$ be the ideal of all closed sets in $X$. Then $C_\mathscr{P}(X)=C(X)$ and $X_\mathscr{P}= \mathbb{R}$ which is not a disconnected subspace of $X$. However, $dt(AG(C_\mathscr{P}(X)))=2$, because if we consider $f(x)=\begin{cases}
		0&x\leq 0\\
		x&x\geq 0
	\end{cases}$ and $g(x)=\begin{cases}
		x&x\leq 0\\
		0&x\geq 0
	\end{cases}$, then $f,g\in Z(C(X))^*$ with $f.g=0$ and $int_XZ(f)\cap int_XZ(g)=\emptyset$. So, by Theorem~\ref{Th4.23}, $dt(AG(C(\mathbb{R})))=2$.\\

In the following theorem, the dominating numbers of the annihilator graphs of $C(X)$, $C_K(X)$ and $C_F(X)$ are recorded.
\begin{theorem}\label{dom}
For a topological space $X$,
\begin{enumerate}
\item $dt(AG(C(X))) = 2$, if $X$ is disconnected. \\
If $X$ is connected then $dt(AG(C(X))) = 2$ if and only if $X$ is not a middle $P$-space.
\item $dt(AG(C_K(X))) = 2$, if $X$ is a locally compact and disconnected space.
\item $dt(AG(C_F(X))) =\begin{cases} 2, & \textnormal{ if } K_X \textnormal{ is finite}\cr \aleph_0, & \textnormal{ otherwise}\end{cases}$.
\end{enumerate}
\end{theorem}
\begin{proof}
\begin{enumerate}
\item If $X$ is disconnected then $\Gamma(C(X))$ is not hypertriangulated and therefore, by Theorem \ref{Th4.26}, $dt(AG(C(X)))=2$.\\
If $X$ is connected then $X$ is not a middle $P$-space if and only if $\Gamma(C(X))$ is not hypertriangulated \cite{Azarpanah}. So, the result follows from Theorem \ref{Th4.26}.
\item If $X$ is locally compact then $X = X_\mathscr{P}$ and so, by disconnectedness of $X$, we get by Theorem \ref{Th4.25} that $dt(AG(C_K(X)))=2$.
\item If $K_X$, the set of isolated points of $X$ is finite then it is a disconnected subset of $X$ containing proper clopen subsets. Therefore, by Theorem \ref{Th4.25}, $dt(AG(C_F(X)))=2$. \\
Let $dt(AG(C_F(X)))$ be infinite. Choosing a countably infinite subset $K$ of $K_X$, we set $W=\{1_x\in Z(C_F(X))^*:x\in K\}$. We claim that $W$ is a dominating set in $AG(C_F(X))$. If possible, let there be $f\in Z(C_F(X))^*\setminus W$ such that $f$ is not adjacent to any $1_x$, $x\in K$. i.e., For each $x\in K$, either $Z(f) \subseteq Z(1_x)$ or $Z(1_x) \subseteq Z(f)$. In any case, $K \subseteq X \setminus Z(f)$ which leads to a contradiction, as $X\setminus Z(f)$ is finite. Hence, $dt(AG(C_F(X)))\leq |K|=\aleph_0$. \\
It is now enough to show that no finite set of vertices is a dominating set in $AG(C_F(X))$. Let $W'=  \{f_i\in Z(C_F(X))^*:i=1,2,...,n\}$. Choose $x\in K_X\setminus \left(\bigcup\limits_{i=1}^n X\setminus Z(f_i)\right)$. Such a choice of $x$ is possible, as $\bigcup\limits_{i=1}^n X\setminus Z(f_i)$ is a finite subset of $K_X$ and $K_X$ is infinite. Considering $f = 1_A$ where $A = \{x\} \cup \left(\bigcup\limits_{i=1}^n X\setminus Z(f_i)\right)$ we observe that $f \in Z(C_F(X))^*$, however, $f$ is not adjacent to any $f_i \in W'$. \\
Hence, $dt(AG(C_F(X))) = \aleph_0$. 
\end{enumerate}
\end{proof}
The observation that the vertices $f$ and $2f$ are non-adjacent in  $AG(C_\mathscr{P}(X))$ leads to the result : $dt(AG(C_\mathscr{P}(X))) = 2$ if and only if $dt_t(AG(C_\mathscr{P}(X))) = 2$. In the next theorem, we generalize this result if $dt(AG(C_\mathscr{P}(X))) < c$(= the cardinality of the continuum).
\begin{theorem}
		If $dt(AG(C_\mathscr{P}(X)))<c$ then $dt(AG(C_\mathscr{P}(X)))=dt_t(AG(C_\mathscr{P}(X)))$.
	\end{theorem}
	\begin{proof}
		Let $H$ be a dominating set in $AG(C_\mathscr{P}(X))$ such that $|H|=dt(AG(C_\mathscr{P}(X)))$ $< c$. If $H$ is not a total dominating set of $AG(C_\mathscr{P}(X))$, there exists $f\in H$ such that $f$ is not adjacent to any vertices of $H$. So, $r.f$ is not adjacent to any vertices of $H$, for each $r\in\mathbb{R}\setminus\{0\}$. As a result $\{r.f : r\in\mathbb{R}\setminus\{0\} \} \subset H$ which gives $ |H|\geq c$. Therefore, $H$ is a total dominating set of $AG(C_\mathscr{P}(X))$ and hence, $dt_t(AG(C_\mathscr{P}(X)))\leq |H|= dt(AG(C_\mathscr{P}(X)))$. The reverse implication follows from the definition.
	\end{proof}

\section{An induced subgraph of $C_\mathscr{P}(X)$}
In this section, we define an equivalence relation on the set of vertices of $AG(C_\mathscr{P}(X))$ and fetch exactly one element from each equivalence class to constitute a set of vertices for a  subgraph of it. The adjacency relation on the new set of vertices is defined accordingly so that it becomes an induced subgraph $G(C_\mathscr{P}(X))$ of $AG(C_\mathscr{P}(X))$. As proposed in the introduction, we observe that the graph properties of $G(C_\mathscr{P}(X))$ are mostly analogous to those of $AG(C_\mathscr{P}(X))$.\\		
A relation $\sim$ on $C_\mathscr{P}(X)$ is defined as follows: For all $f,g\in C_\mathscr{P}(X)$,
$$f\sim g \textnormal{ if and only if } int_XZ(f)\cap X_\mathscr{P}=int_XZ(g)\cap X_\mathscr{P} .$$ 
$\sim$ being an equivalence relation on $C_\mathscr{P}(X)$, partitions the set in disjoint equivalence classes. In what follows, $[f]$ denotes the equivalence class containing $f$. The following observations are useful in the subsequent development of this paper:
\begin{enumerate}
\item[(a)] No two members of $[f]$ are adjacent in $AG(C_\mathscr{P}(X))$, for each $f\in C_\mathscr{P}(X)$.
\item[(b)] $[0] = \{0\}$. For if $g \in [0]$, $int_XZ(g)\cap X_\mathscr{P} = X_\mathscr{P}$ and $X \setminus Z(g) \subset X_\mathscr{P}$ which imply $g=0$.
\item[(c)] For any $f$, $g \in Z(C_\mathscr{P}(X))^*$, $f$ is adjacent to $g \implies$ $f$ is adjacent to $h$, for all $h\in [g]$.
\end{enumerate}
Choosing exactly one element from each equivalence class, a subset $S$ of $C_\mathscr{P}(X)$ is constructed. The induced subgraph obtained by taking the set of vertices as $V(C_\mathscr{P}(X)) = \{f\in S\setminus\{0\}:int_XZ(f)\cap X_\mathscr{P}\neq\emptyset\}\subset Z(C_\mathscr{P}(X))^*$, is denoted by $G(C_\mathscr{P}(X))$. By definition of induced subgraph, for any $f, g\in V$,
$f$ is adjacent to $g$ in $G(C_\mathscr{P}(X))$ if and only if $f$ is adjacent to $g$ in $AG(C_\mathscr{P}(X))$, i.e., $int_XZ(g)\cap X_\mathscr{P}\not\subset int_XZ(f)\cap X_\mathscr{P}$ and $int_XZ(f)\cap X_\mathscr{P}\not\subset int_XZ(g)\cap X_\mathscr{P}$. Further,
\begin{enumerate}
\item[(a)] $f=g$ in $V(C_\mathscr{P}(X))$ if and only if $int_XZ(f)\cap X_\mathscr{P}=int_XZ(g)\cap X_\mathscr{P}$.
\item[(b)] For each $f\in Z(C_\mathscr{P}(X))^*$ there exists a unique $f'\in V(C_\mathscr{P}(X))$ such that $[f']=[f]$.
\end{enumerate}
A useful result based on these observations is recorded in the following lemma for future references.
\begin{lemma}\label{Lem5.1}
For each edge $f-g$ in $AG(C_\mathscr{P}(X))$ there exists a unique edge $f'-g'$ in $G(C_\mathscr{P}(X))$ such that $[f']=[f]$ and $[g']=[g]$.
\end{lemma}
Before discussing various graph parameters of the induced subgraph of $AG(C_\mathscr{P}(X))$, we cite a few illustrative examples of $G(C_\mathscr{P}(X))$ and visualize them when $|X_\mathscr{P}|$ is finite. 

\begin{example}
Let $|X_\mathscr{P}| = 2$. If $X_\mathscr{P}= \{x, y\}$ then any vertex $f$ in $AG(C_\mathscr{P}(X))$ ($=\Gamma (C_\mathscr{P}(X))$) takes the form $r.1_x$ or $r.1_y$, for some $r\in \mathbb{R}\setminus \{0\}$, where $1_A$ denotes the characteristic function on $A\subseteq X$. Therefore, $Z(C_\mathscr{P}(X))^*=\{r.1_x:r\in\mathbb{R}\setminus\{0\}\}\cup\{r.1_y:r\in\mathbb{R}\setminus\{0\}\}$ and hence $AG(C_\mathscr{P}(X))$ is a complete bipartite graph as shown below:
	
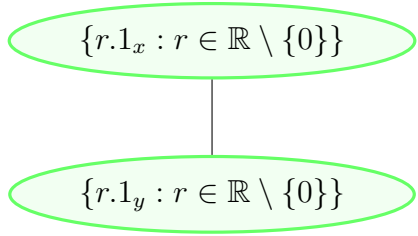
\begin{center}
		\begin{tikzpicture}[roundnode/.style={ellipse, draw=green!60, fill=green!5, very thick, minimum size=7mm}]
			\node[roundnode] (maintopic) {$\{r.1_x:r\in\mathbb{R}\setminus\{0\}\}$};
			\node[roundnode] (lowercircle) [below=of maintopic] {$\{r.1_y:r\in\mathbb{R}\setminus\{0\}\}$};
			\draw (maintopic.south) -- (lowercircle.north);
		\end{tikzpicture}
	\end{center} 
Choosing $V = \{1_x, 1_y\}$, we obtain $G(C_\mathscr{P}(X))$ as the complete graph $K_2$ with two vertices, as the following:
		
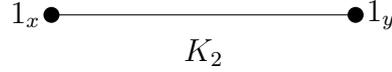
\begin{center}
			\begin{tikzpicture}
				\draw (0,0)--(4,0);
				\fill (0,0) circle (3pt) node[left] {$1_x$};
				\fill (4,0) circle (3pt) node[right] {$1_y$};
				\node (start) at (2,-0.5) {$K_2$};
			\end{tikzpicture}
		\end{center}
\end{example}
\begin{example}
Let $|X_\mathscr{P}| = 3$. Then $f\in Z(C_\mathscr{P}(X))^*$ if and only if either $|X\setminus Z(f)|=1$ or $2$. If $X_\mathscr{P}=\{x,y,z\}$ then $X\setminus Z(f)$ is one of $\{x\}$, $\{y\}$, $\{z\}$, $\{x,y\}$, $\{x,z\}$, $\{y,z\}$. Therefore, $AG(C_\mathscr{P}(X))$ is a $6$-partite graph of the following form :
	
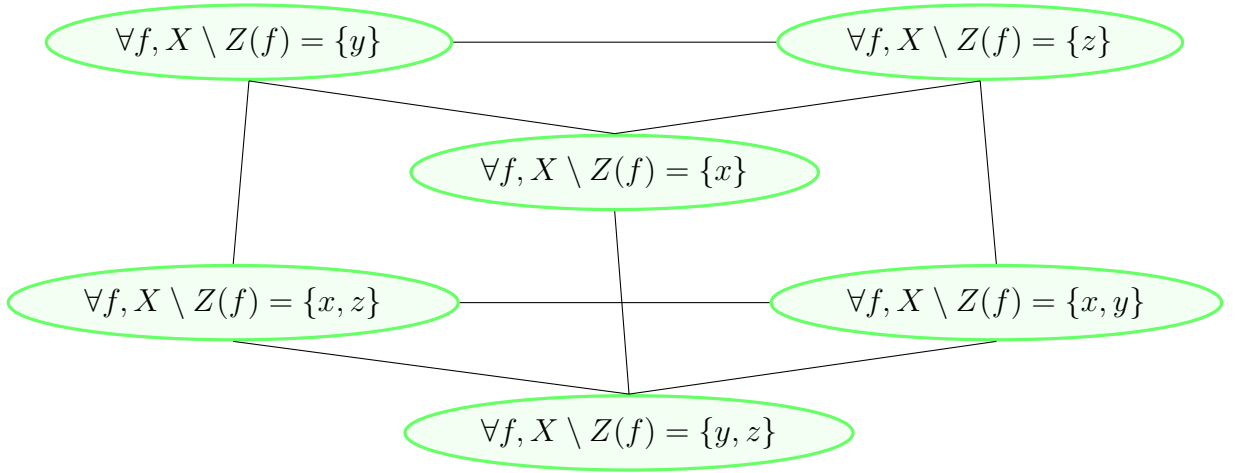
\begin{center}
		\begin{tikzpicture}[roundnode/.style={ellipse, draw=green!60, fill=green!5, very thick, minimum size=7mm}]
			\node[roundnode] (maintopic) {$\forall f, X\setminus Z(f)=\{x\}$};
			\node[roundnode] (upperleftcircle) [above left=of maintopic] {$\forall f, X\setminus Z(f)=\{y\}$};
			\node[roundnode] (upperrightcircle) [above right=of maintopic] {$\forall f, X\setminus Z(f)=\{z\}$};
			\node[roundnode] (lowerleftcircle) [below left=of maintopic] {$\forall f, X\setminus Z(f)=\{x,z\}$};
			\node[roundnode] (lowerrightcircle) [below right=of maintopic] {$\forall f, X\setminus Z(f)=\{x,y\}$};
			\node[roundnode] (lowercircle) [below right=of lowerleftcircle] {$\forall f, X\setminus Z(f)=\{y,z\}$};
			\draw (upperleftcircle.south) -- (maintopic.north);
			\draw (upperrightcircle.south) -- (maintopic.north);
			\draw (upperleftcircle.east) -- (upperrightcircle.west);
			\draw (maintopic.south) -- (lowercircle.north);
			\draw (upperleftcircle.south) -- (lowerleftcircle.north);
			\draw (upperrightcircle.south) -- (lowerrightcircle.north);
			\draw (lowerleftcircle.east) -- (lowerrightcircle.west);
			\draw (lowerleftcircle.south) -- (lowercircle.north);
			\draw (lowerrightcircle.south) -- (lowercircle.north);
		\end{tikzpicture}
	\end{center}
Considering $V=\{1_x,1_y,1_z, 1_{\{x,y\}}, 1_{\{x,z\}},1_{\{y,z\}}\}$, the induced subgraph $G(C_\mathscr{P}(X))$ takes the form as shown below:
		
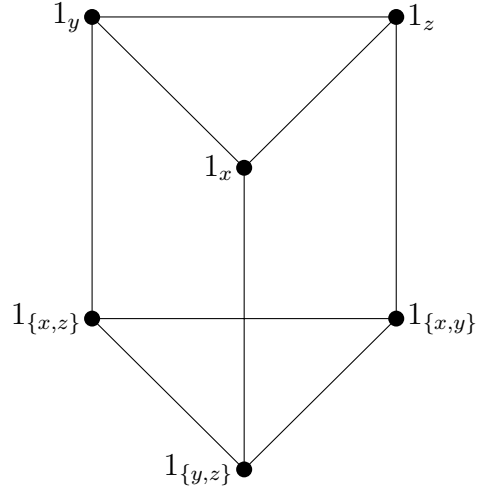
\begin{center}
			\begin{tikzpicture}
				\draw (2,2)--(-2,2)--(0,0)--(2,2) (-2,-2)--(0,-4)--(2,-2)--(-2,-2) (-2,2)--(-2,-2) (2,2)--(2,-2) (0,0)--(0,-4);
				\fill (0,0) circle (3pt) node[left] {$1_x$};
				\fill (-2,2) circle (3pt) node[left] {$1_y$};
				\fill (2,2) circle (3pt) node[right] {$1_z$};
				\fill (0,-4) circle (3pt) node[left] {$1_{\{y,z\}}$};
				\fill (-2,-2) circle (3pt) node[left] {$1_{\{x,z\}}$};
				\fill (2,-2) circle (3pt) node[right] {$1_{\{x,y\}}$};
			\end{tikzpicture}
		\end{center}
\end{example}
Scrutinising the pattern of getting $G(C_\mathscr{P}(X))$ from $AG(C_\mathscr{P}(X))$ in the above examples, one expects that if $|X_\mathscr{P}| = \ finite$ then $V = \{1_A \ : \ \emptyset\neq A\subsetneqq X_\mathscr{P}\}$ serves as the set of vertices of $G(C_\mathscr{P}(X))$ and two vertices $1_A, 1_B$ are adjacent in $G(C_\mathscr{P}(X))$ if and only if $A\not\subset B$ and $B\not\subset A$. This can be deduced as a special case of a more general example (where $X_\mathscr{P}$ need not be finite) as cited below.  
\begin{example}\label{Ex7.1}
Consider $\mathcal{P}$ as the ideal of all finite subsets of $X$. Then $C_\mathscr{P}(X)=C_F(X)$ and $X_\mathscr{P}=K_X$, the set of all isolated points of $X$. If $f\in C_F(X)$ then $X\setminus Z(f)$ is finite and therefore, $[f]=[1_{X\setminus Z(f)}]$. So, considering $V=\{1_A:\emptyset\neq A\subsetneqq K_X \text{ and }A\text{ is finite}\}$, the  induced subgraph $G(C_F(X))$ of $AG(C_F(X))$ is achieved. Clearly, two vertices $1_A, 1_B$ are adjacent in $G(C_F(X))$ if and only if $A\not\subset B$ and $B\not\subset A$.
\end{example}
It is relevant to inquire when the induced subgraph is finite. The following theorem asserts that the finiteness of $G(C_\mathscr{P}(X))$ completely depends upon the cardinality of $X_\mathscr{P}$.
 
\begin{theorem}\label{Th7.3}
$G(C_\mathscr{P}(X))$ is a finite graph if and only if $X_\mathscr{P}$ is a finite set. \\
Moreover, if $X_\mathscr{P}$ is finite, then $|G(C_\mathscr{P}(X))|=2^{|X_\mathscr{P}|}-2$ and $deg(f)=2^{|X_\mathscr{P}|} -2^{|X\setminus Z(f)|}-2^{|X_\mathscr{P}\cap Z(f)|}+1$ for each $f\in V(C_\mathscr{P}(X))$.
\end{theorem}
\begin{proof}
Let  $x,y\in X_\mathscr{P}$. Then by Lemma \ref{Lem1}, there exist $f_x,f_y\in C_\mathscr{P}(X)$ such that $x\in X\setminus Z(f_x)\subset X\setminus int_XZ(f_x)\subset X\setminus \{y\}$ and $y\in X\setminus Z(f_y)\subset X\setminus int_XZ(f_y)\subset X\setminus \{x\}$. Then $y\in int_XZ(f_x)\cap X_\mathscr{P}$, $x\in X\setminus Z(f_x)$ and $x\in int_XZ(f_y)\cap X_\mathscr{P}$, $y\in X\setminus Z(f_y)$. So, $f_x, f_y\in Z(C_\mathscr{P}(X))^*$ such that $[f_x]\neq [f_y]$. Let $f'_x, f'_y\in V(C_\mathscr{P}(X))$ be such that $[f'_x]=[f_x]$ and $[f'_y]=[f_y]$. Then $x \mapsto f'_x$ defines an injective function from $X_\mathscr{P}$ to $G(C_\mathscr{P}(X))$. Hence,  $|G(C_\mathscr{P}(X))| \geq |X_\mathscr{P}|$. \\
If $X_\mathscr{P}$ is infinite then $G(C_\mathscr{P}(X))$ is an infinite graph. In other words, if $G(C_\mathscr{P}(X))$ is a finite graph then $X_\mathscr{P}$ is a finite set. \\
Conversely, if $X_\mathscr{P}$ is finite then by a special case of Example \ref{Ex7.1}, we get $V=\{1_A:\emptyset\neq A\subsetneqq X_\mathscr{P}\}$ and so, $|V|= |\mathcal{P}(X_\mathscr{P})|-2$, which is finite. Therefore, $G(C_\mathscr{P}(X))$ is a finite graph. \\
Also, $|G(C_\mathscr{P}(X))|=|\mathcal{P}(X_\mathscr{P})|-2=2^{|X_\mathscr{P}|}-2$. \\
Let $f\in V$. Then $f=1_A$ where $A=X\setminus Z(f)$.If $1_B\in V$ is not adjacent to $1_A$ then either $B\subset A$ or $A\subset B$. Now $\{1_B\in V : B\subset A\}=\mathcal{P}(A)\setminus \{\emptyset\}\implies|\{1_B\in V : B\subset A\}|= |\mathcal{P}(A)\setminus \{\emptyset\}|=2^{|A|}-1$ and $\{1_B\in V : A\subsetneqq B\}=\{A\cup C: C\in\mathcal{P}(X_\mathscr{P}\setminus A)\setminus \{\emptyset, X_\mathscr{P}\setminus A\}\}\implies |\{1_B\in V: A\subsetneqq B\}|=|\mathcal{P}(X_\mathscr{P}\setminus A)\setminus \{\emptyset, X_\mathscr{P}\setminus A\}|=2^{|X_\mathscr{P}\setminus A|}-2$. Therefore, $\{1_B\in V :1_B\text{ is not adjacent to }1_A\text{ in }G(C_\mathscr{P}(X))\}=\{1_B\in V : B\subsetneqq A\}\sqcup\{1_B\in V : A\subsetneqq B\}\implies |\{1_B\in V : 1_B\text{ is not adjacent to }1_A\text{ in }G(C_\mathscr{P}(X))\}|= (2^{|A|}-1)+(2^{|\mathcal{P}(X_\mathscr{P}\setminus A)|}-2)=2^{|A|}+2^{|X_\mathscr{P}\setminus A|}-3$. Hence,
\begin{eqnarray*}
deg(f) = deg(1_A) &=&|\{1_B\in V \ : \ 1_B\text{ is adjacent to }1_A\text{ in }G(C_\mathscr{P}(X))\}|\\ 
&=&|V|-|\{1_B\in V\ : \ 1_B\text{ is not adjacent to}1_A\text{ in }G(C_\mathscr{P}(X))\}|\\ 
&=& (2^{|X_\mathscr{P}|}-2)-(2^{|A|}+2^{|X_\mathscr{P}\setminus A|}-3) \\ 
&=& 2^{|X_\mathscr{P}|}- 2^{|A|}-2^{|X_\mathscr{P}\setminus A|}+1.
\end{eqnarray*}
Since $A=X\setminus Z(f)$, $X_\mathscr{P}\setminus A=X_\mathscr{P}\cap Z(f)$ and consequently, $deg(f)=2^{|X_\mathscr{P}|} -2^{|X\setminus Z(f)|}-2^{|X_\mathscr{P}\cap Z(f)|}+1$, for each $f\in V(C_\mathscr{P}(X))$.
\end{proof}

We now establish that $AG(C_\mathscr{P}(X))$ and $G(C_\mathscr{P}(X))$ have same values for certain important graph parameters, namely, diameter, eccentricity and girth.
\begin{theorem}
			$d_G(f,g)=d_{AG}(f,g)$ for each $f,g\in V(C_\mathscr{P}(X))$.
		\end{theorem}
		\begin{proof}
			Let $f,g\in V(C_\mathscr{P}(X))$ be not adjacent in $V(C_\mathscr{P}(X))$. Then $f,g$ are not adjacent in $AG(C_\mathscr{P}(X))$ and so, $d_{AG(C_\mathscr{P}(X))}(f,g)=2$ by Corollary \ref{Cor4.7}, i.e., there exists $h\in Z(C_\mathscr{P}(X))^*$ which is adjacent to both $f,g$ in $AG(C_\mathscr{P}(X))$ . From Theorem \ref{Th4.2}, it follows that $int_XZ(h)\cap X_\mathscr{P}\neq int_XZ(f)\cap X_\mathscr{P}$ and $int_XZ(h)\cap X_\mathscr{P}\neq int_XZ(g)\cap X_\mathscr{P}$, i.e., $[h]\neq [f]$ and $[h]\neq [g]$ and so, we can consider $h\in V(C_\mathscr{P}(X))$. Thus $h$ is adjacent to both $f,g$ in $G(C_\mathscr{P}(X))\implies d_{G}(f,g)=2$. Thus $d_{G}(X)(f,g)=d_{AG}(f,g)$ for each $f,g\in V(C_\mathscr{P}(X))$.
		\end{proof}
		\begin{corollary}\label{Cor5.3}
			$diam(G(C_\mathscr{P}(X)))=diam(AG(C_\mathscr{P}(X)))=2$
		\end{corollary}
		\begin{theorem}\label{Th7.6}		  	   
			$ecc_{G}(f)=2=ecc_{AG}(f)$ for each $f\in V(C_\mathscr{P}(X))$.
		\end{theorem}
		\begin{proof}
			Let $f\in V(C_\mathscr{P}(X))$. Then $X\setminus Z(f)\neq\emptyset$.\\
If $|X\setminus Z(f)|=1$, say $X\setminus Z(f)=\{x\}$ then $int_XZ(f)\cap X_\mathscr{P}=X_\mathscr{P}\setminus\{x\}$.\\
			Since $|X_\mathscr{P}|\geq 3$, there exist distinct $y,z\in X_\mathscr{P}\setminus\{x\}$. By Lemma \ref{Lem2}, we find $g\in C_\mathscr{P}(X)$ such that $\{x,y\}\subset X\setminus Z(g)\subset X\setminus int_XZ(g)\subset X\setminus\{z\}$. Then $z\in int_XZ(g)\cap X_\mathscr{P}$ and $x,y\notin Z(g)\implies g\in Z(C_\mathscr{P}(X))^*$ and $int_XZ(g)\cap X_\mathscr{P}\subsetneqq int_XZ(f)\cap X_\mathscr{P}$, i.e., $[f]\neq [g]$. By Lemma~\ref{Lem5.1}, there exists $g'\in V(C_\mathscr{P}(X))$ such that $[g']=[g]$. Then $f,g'\in  V(C_\mathscr{P}(X))$ with $int_XZ(g')\cap X_\mathscr{P}\subsetneqq int_XZ(f)\cap X_\mathscr{P}\implies f,g'$ are not adjacent in $G(C_\mathscr{P}(X))$.\\
Let $|X\setminus Z(f)|\geq 2$ and choose distinct $x,y\in X\setminus Z(f)$.\\
			By Lemma \ref{Lem1}, there exists $g\in C_\mathscr{P}(X)$ such that $x\in X\setminus Z(g)\subset X\setminus int_XZ(g)\subset X\setminus[Z(f)\cup \{y\}]$. Then $g\neq 0$ and $Z(f)\cup\{y\}\subset int_XZ(g)\implies g\in Z(C_\mathscr{P}(X))^*$ and $int_XZ(f)\cap X_\mathscr{P}\subsetneqq int_XZ(g)\cap X_\mathscr{P}$, i.e., $[f]\neq [g]$. As done earlier, considering $g\in V(C_\mathscr{P}(X))$ we find that $f,g$ are not adjacent in $G(C_\mathscr{P}(X))$.\\
			Therefore, in any case, for each $f\in V(C_\mathscr{P}(X))$, there exists $g\in V(C_\mathscr{P}(X))$ such that $g$ is not adjacent to $f$ in $G(C_\mathscr{P}(X))$. So, $ecc_{G}(f)\geq 2$. As $diam(G(C_\mathscr{P}(X)))=2$, it follows that $ecc_{G}(f)=2$ ($=ecc_{AG}(f)$), for each $f\in V(C_\mathscr{P}(X))$.
		\end{proof}
	\begin{corollary}
\begin{enumerate}
	\item[(i)] $C(G(C_\mathscr{P}(X)))=V(C_\mathscr{P}(X))$
\item[(ii)] $rad(G(C_\mathscr{P}(X))) =2=rad(AG(C_\mathscr{P}(X)))$.
\end{enumerate}
	\end{corollary}
	\begin{theorem}\label{Th5.6}
		Let $f,g\in V(C_\mathscr{P}(X))$. 
		\begin{enumerate}
			\item If $f.g\neq 0$ then there always exists a vertex in $V(C_\mathscr{P}(X))$ which is adjacent to both $f,g$ in $G(C_\mathscr{P}(X))$.
			\item If $f.g=0$ then there exists a vertex in $V(C_\mathscr{P}(X))$ which is adjacent to both $f,g$ in $G(C_\mathscr{P}(X))$ if and only if $|int_XZ(f)\cap X_\mathscr{P}|\geq 2$ and $|int_XZ(g)\cap X_\mathscr{P}|\geq 2$.
		\end{enumerate}
	\end{theorem}
	\begin{proof}
		\hspace*{3cm}
		\begin{enumerate}
			\item If $f.g\neq 0$ then by Theorem~\ref{Th4.6}, there always exists a vertex $h$ in $Z(C_\mathscr{P}(X))^*$ such that $h$ is adjacent to both $f$ and $g$ in $AG(C_\mathscr{P}(X))$. So, $[h]\neq [f]$ and $[h]\neq [g]$. Let $h'\in V(C_\mathscr{P}(X))$ such that $[h']=[h]$. Clearly, $h'$ is adjacent to both $f,g$ in $G(C_\mathscr{P}(X))$.
			\item Let $f.g=0$. If $|int_XZ(f)\cap X_\mathscr{P}|\geq 2$ and $|int_XZ(g)\cap X_\mathscr{P}|\geq 2$, then by Theorem~\ref{Th4.6}(ii), there exists a vertex $h$ in $Z(C_\mathscr{P}(X))^*$ which is adjacent to both $f,g$ in $AG(C_\mathscr{P}(X))$. Let $h'\in V(C_\mathscr{P}(X))$ such that $[h']=[h]$. Then, $h$ is adjacent to both $f,g$ in $G(C_\mathscr{P}(X))$. Conversely let $h$ be adjacent to both $f,g$ in $G(C_\mathscr{P}(X))$. Then $h$ is adjacent to both $f,g$ in $AG(C_\mathscr{P}(X))$ also and hence by Theorem~\ref{Th4.6}(ii), $|int_XZ(f)\cap X_\mathscr{P}|\geq 2$ and $|int_XZ(g)\cap X_\mathscr{P}|\geq 2$.
		\end{enumerate}
	\end{proof}
	\begin{theorem}
		$G(C_\mathscr{P}(X))$ is hypertriangulated if and only if $AG(C_\mathscr{P}(X))$ is hypertriangulated.
	\end{theorem}
	\begin{proof}
		Let $G(C_\mathscr{P}(X))$ be hypertriangulated and $f-g$ be an edge in $AG(C_\mathscr{P}(X))$. Then by Lemma~\ref{Lem5.1} there exists a unique edge $f'-g'$ in $G(C_\mathscr{P}(X))$ such that $[f']=[f]$ and $[g']=[g]$. Since $G(C_\mathscr{P}(X))$ is hypertriangulated, there exists $h\in V(C_\mathscr{P}(X))$ such that $h$ is adjacent to both $f',g'$ in $G(C_\mathscr{P}(X))$. So,  $h$ is adjacent to both $f,g$ in $AG(C_\mathscr{P}(X))$ and hence $f-g-h-f$ is a triangle in $AG(C_\mathscr{P}(X))$, proving that $AG(C_\mathscr{P}(X))$ is hypertriangulated.\\
		Conversely, let $f-g$ be an edge in $G(C_\mathscr{P}(X))$. Then $f-g$ is an edge in $AG(C_\mathscr{P}(X))$. So, there exists $h\in Z(C_\mathscr{P}(X))^*$ such that $f-g-h-f$ is a triangle in $AG(C_\mathscr{P}(X))$. Consider $h'\in V(C_\mathscr{P}(X))$ such that $[h']=[h]$. Then $f-g-h'-f$ is a triangle in $G(C_\mathscr{P}(X))$. Therefore, $G(C_\mathscr{P}(X))$ is hypertriangulated.
	\end{proof}
	\begin{theorem}
		$G(C_\mathscr{P}(X))$ is always triangulated.
	\end{theorem}
	\begin{proof}
		Let $f\in V(C_\mathscr{P}(X))$. Since $AG(C_\mathscr{P}(X))$ is always triangulated and $f\in Z(C_\mathscr{P}(X))^*$, there exists a triangle in $AG(C_\mathscr{P}(X))$ containing $f$ as a vertex, say $f-g-h-f$ for some $g,h\in Z(C_\mathscr{P}(X))^*$. Let $g',h'\in V(C_\mathscr{P}(X))$ be such that $[g']=[g]$ and $[h']=[h]$. Hence, $f-g'-h'-f$ is a triangle in $G(C_\mathscr{P}(X))$, proving that $G(C_\mathscr{P}(X))$ is triangulated.
	\end{proof}
	\begin{corollary}
		$gr(G(C_\mathscr{P}(X)))=gr(AG(C_\mathscr{P}(X)))$.
	\end{corollary}
The following theorem follows immediately from Theorem~\ref{Th5.6} :
\begin{theorem}\label{Th5.10}
		For $f,g\in V(C_\mathscr{P}(X))$, $f\perp_{G} g$ if and only if $f.g=0$ and either $|int_XZ(f)\cap X_\mathscr{P}|=1$ or $|int_XZ(g)\cap X_\mathscr{P}|=1$.
	\end{theorem}
\begin{theorem}
		For $f,g\in Z(C_\mathscr{P}(X))^*$, there exist unique $f',g'\in V(C_\mathscr{P}(X))$ such that $f\perp_{AG} g$ if and only if $f'\perp_{G} g'$.
	\end{theorem}
	\begin{proof} 
		For $f,g\in Z(C_\mathscr{P}(X))^*$, there exist unique $f',g'\in V(C_\mathscr{P}(X))$ such that  $[f']=[f]$ and $[g]=[g']$.  $[f']=[f] \implies$ $int_XZ(f)\cap X_\mathscr{P}=int_XZ(f')\cap X_\mathscr{P}$ and similarly, $[g]=[g'] \implies$ $int_XZ(g)\cap X_\mathscr{P}=int_XZ(g')\cap X_\mathscr{P}$. Then $f.g=0\iff X\setminus Z(f)\subset int_XZ(g)\cap X_\mathscr{P}=int_XZ(g')\cap X_\mathscr{P}\iff f.g'=0\iff X\setminus Z(g')\subset int_XZ(f)\cap X_\mathscr{P}=int_XZ(f')\cap X_\mathscr{P}\iff f'.g'=0$. Therefore, 
		\begin{equation*}
			\begin{split}
				f\perp_{AG} g &\iff f.g=0\text{ and either }|int_XZ(f)\cap X_\mathscr{P}|=1\\
				&\qquad\qquad\quad\qquad\text{ or }|int_XZ(g)\cap X_\mathscr{P}|=1\\
				&\iff f'.g'=0\text{ and either }|int_XZ(f')\cap X_\mathscr{P}|=1\\
				&\qquad\qquad\quad\qquad\text{ or }|int_XZ(g')\cap X_\mathscr{P}|=1\\
				&\iff f'\perp_{G(C_\mathscr{P}(X))} g', \text{ by Theorem \ref{Th5.10}}
			\end{split}
		\end{equation*}
	\end{proof}
	\begin{corollary}\label{Cor5.12}
		For $f,g\in V(C_\mathscr{P}(X))$, $f\perp_{G}g$ if and only if $f\perp_{AG}g$.
	\end{corollary}

\begin{theorem}\label{Th5.12}
		If $G(C_\mathscr{P}(X))$ is complemented then it is uniquely complemented.
	\end{theorem}
	\begin{proof}
		Let $f\perp_G g$ and $f\perp_G h$, for some $f,g,h\in V(C_\mathscr{P}(X))$. By Theorem \ref{Th5.10}, $f.g=0=f.h$ and either $|int_XZ(f)\cap X_\mathscr{P}|=1$ or $|int_XZ(g)\cap X_\mathscr{P}|=1=|int_XZ(h)\cap X_\mathscr{P}|$.\\
If $|int_XZ(f)\cap X_\mathscr{P}|=1$, say $int_XZ(f)\cap X_\mathscr{P}=\{x\}$, then
$f.g=0\implies X\setminus Z(g)\subset int_XZ(f)\cap X_\mathscr{P}=\{x\}$, so that  $int_XZ(g)\cap X_\mathscr{P}=X_\mathscr{P}\setminus\{x\}$.  Similarly, $f.h=0\implies int_XZ(h)\cap X_\mathscr{P}=X_\mathscr{P}\setminus\{x\}$. So, $int_XZ(g)\cap X_\mathscr{P}=int_XZ(h)\cap X_\mathscr{P}$.\\
If $|int_XZ(g)\cap X_\mathscr{P}|=1=|int_XZ(h)\cap X_\mathscr{P}|$ then assuming $int_XZ(g)\cap X_\mathscr{P}=\{x\}$ and $int_XZ(h)\cap X_\mathscr{P}=\{y\}$ we get $X\setminus Z(f)=\{x\}$ and $X\setminus Z(f)=\{y\}$ respectively from $f.g=0$ and $f.h = 0$. Therefore, $x=y$ and so, $int_XZ(g)\cap X_\mathscr{P}=int_XZ(h)\cap X_\mathscr{P}$.\\
		i.e., $int_XZ(g)\cap X_\mathscr{P}=int_XZ(h)\cap X_\mathscr{P}$ holds in any case. So, any $k\in V(C_\mathscr{P}(X))$ is adjacent to $g$ in $G(C_\mathscr{P}(X))$ if and only if $k$ is adjacent to $h$ in $G(C_\mathscr{P}(X))$. i.e., $g,h$ are adjacent to the same set of vertices in $G(C_\mathscr{P}(X))$ which proves that $G(C_\mathscr{P}(X))$ is uniquely complemented.
	\end{proof}
\begin{theorem}
		$G(C_\mathscr{P}(X))$ is uniquely complemented if and only if $AG(C_\mathscr{P}(X))$ is uniquely complemented.
	\end{theorem}
	\begin{proof}
		In view of Theorem \ref{Th4.20} and \ref{Th5.12}, it is enough to show that $G(C_\mathscr{P}(X))$ is complemented if and only if $AG(C_\mathscr{P}(X))$ is complemented. \\
Let $G(C_\mathscr{P}(X))$ be complemented and $f\in Z(C_\mathscr{P}(X))^*$. Let $f'\in V(C_\mathscr{P}(X))$ be such that $[f']=[f]$. Since $G(C_\mathscr{P}(X))$ is complemented, there exists $g\in V(C_\mathscr{P}(X))$ such that $f\perp_{G}g$ and hence by Corollary \ref{Cor5.12}, $f\perp_{AG}g$, proving that $AG(C_\mathscr{P}(X))$ is complemented.\\
		Conversely, let $f\in V(C_\mathscr{P}(X))$. Since $AG(C_\mathscr{P}(X))$ is complemented, there exists $g\in Z(C_\mathscr{P}(X))^*$ such that $f\perp_{AG}g$. By Theorem \ref{Th5.12}, $f\perp_{AG}g'$ where $g'\in V(C_\mathscr{P}(X))$ and $[g']=[g]$. Therefore, $G(C_\mathscr{P}(X))$ is complemented.
	\end{proof}

\begin{lemma}\label{Lem5.17}
		For each dominating set $W$ of $AG(C_\mathscr{P}(X))$, $W'=\{f'\in V(C_\mathscr{P}(X)): [f']=[f]\text{ for some }f\in W\}$ is a dominating set of $G(C_\mathscr{P}(X))$ with $|W'|\leq |W|$.
	\end{lemma}
	\begin{proof}
		Clearly, $W\cap V(C_\mathscr{P}(X))\subset W'$. Let $f\in V(C_\mathscr{P}(X))\setminus W'$. Then it follows that $f\in Z(C_\mathscr{P}(X))^*\setminus W$. $W$ being a dominating set in $AG(C_\mathscr{P}(X))$, there exists $g\in W$ such that $g$ is adjacent to $f$ in $AG(C_\mathscr{P}(X))$. Choose $g'\in V(C_\mathscr{P}(X))$ such that $[g']=[g]$. Then $g'\in W'$ and $f,g'$ are adjacent in $G(C_\mathscr{P}(X))$, proving $W'$ to be a dominating set of $G(C_\mathscr{P}(X))$. That $|W'| \leq |W|$ is evident.
	\end{proof}
	\begin{theorem}\label{Th7.16}
		$dt(G(C_\mathscr{P}(X)))\leq dt(AG(C_\mathscr{P}(X)))$.
	\end{theorem}
	\begin{proof}
		Let $W$ be a dominating set in $AG(C_\mathscr{P}(X))$. From Lemma \ref{Lem5.17}, there exists a dominating set $W'$ in $G(C_\mathscr{P}(X))$ such that $|W'|\leq |W|$. So, $dt(G(C_\mathscr{P}(X)))\leq |W'|\leq |W|$. $W$ being an arbitrary dominating set in $AG(C_\mathscr{P}(X))$ it then follows that $dt(G(C_\mathscr{P}(X)))\leq dt(AG(C_\mathscr{P}(X)))$.
	\end{proof}
	\begin{theorem}\label{Th7.19}
		$dt(G(C_\mathscr{P}(X)))=2$ if and only if there exist $f,g\in V(C_\mathscr{P}(X))$ with $f.g=0$ and $int_XZ(f)\cap int_XZ(g)\cap X_\mathscr{P}=\emptyset$. 
	\end{theorem}
	\begin{proof}
		Let there be $f,g\in V(C_\mathscr{P}(X))$ obeying the given conditions. As $f,g\in Z(C_\mathscr{P}(X))^*$, by Theorem \ref{Th4.23}, $dt(AG(C_\mathscr{P}(X)))=2$. Also by Theorem~\ref{Th7.6}, $ecc(f)=2$ for each $f\in V(C_\mathscr{P}(X))$ and therefore, $dt(G(C_\mathscr{P}(X)))>1$. Clubbing all these facts with Theorem~\ref{Th7.16}, we get $dt(G(C_\mathscr{P}(X)))=2$.\\
		Conversely, let $dt(G(C_\mathscr{P}(X)))=2$ and $\{f,g\}\subset V(C_\mathscr{P}(X))$ constitute a dominating set of $G(C_\mathscr{P}(X))$. We claim that $f,g$ are adjacent in $G(C_\mathscr{P}(X))$. If not then either $int_XZ(f)\cap X_\mathscr{P}\subsetneqq int_XZ(g)\cap X_\mathscr{P}$ or $int_XZ(g)\cap X_\mathscr{P}\subsetneqq int_XZ(f)\cap X_\mathscr{P}$. Without loss of generality, we assume that $int_XZ(f)\cap X_\mathscr{P}\subsetneqq int_XZ(g)\cap X_\mathscr{P}$. For $g\in V(C_\mathscr{P}(X))$, $X\setminus Z(g)\neq\emptyset$.\\
If $|X\setminus Z(g)|\geq 2$, say $x,y\in X\setminus Z(g)$, then by Lemma \ref{Lem1}, there exists $h\in C_\mathscr{P}(X)$ such that $x\in X\setminus Z(h)\subset X\setminus int_XZ(h)\subset X_\mathscr{P}\setminus[Z(g)\cup \{y\}]$. Then $Z(g)\cup\{y\}\subset int_XZ(h)$ and $h\neq 0$. Since $y\in [int_XZ(h)\cap X_\mathscr{P}]\setminus Z(g)$, it follows that $int_XZ(g)\cap X_\mathscr{P}\subsetneqq int_XZ(h)\cap X_\mathscr{P}$. Let $h'\in V(C_\mathscr{P}(X))$ with $[h']=[h]$. Therefore, $int_XZ(f)\cap X_\mathscr{P}\subsetneqq int_XZ(g)\cap X_\mathscr{P}\subsetneqq int_XZ(h')\cap X_\mathscr{P}\implies h'\in V(C_\mathscr{P})\setminus\{f,g\}$ and $h$ is not adjacent to $f,g$; this contradicts our hypothesis that$\{f,g\}$ is a dominating set of $G(C_\mathscr{P}(X))$.\\
		If $|X\setminus Z(g)|=1$, say $X\setminus Z(g)=\{x\}\implies int_XZ(g)\cap X_\mathscr{P}=X_\mathscr{P}\setminus\{x\}$ then assume if possible that $X\setminus Z(f)=\{y\}$. Then $int_XZ(f)\cap X_\mathscr{P}=X_\mathscr{P}\setminus\{y\}$. Therefore $int_XZ(f)\cap X_\mathscr{P}\subsetneqq int_XZ(g)\cap X_\mathscr{P}\implies X_\mathscr{P}\setminus\{y\}\subsetneqq X_\mathscr{P}\setminus\{x\}$ which is not possible. Therefore, $|X\setminus Z(f)|\geq 2$ and choose any $y,z\in X\setminus Z(f)$. By Lemma~\ref{Lem2}, there exists $h\in C_\mathscr{P}(X)$ such that $\{x,y\}\subset X\setminus Z(h)\subset X\setminus int_XZ(h)\subset X\setminus[Z(f)\cup \{z\}]$. Therefore, $Z(f)\cup\{z\}\subset int_XZ(h)\subset X\setminus\{x,y\}$ and $h\neq 0$. Let $h'\in V(C_\mathscr{P}(X))$ be such that $[h']=[h]$. Then $int_XZ(f)\cap X_\mathscr{P}\subsetneqq int_XZ(h')\cap X_\mathscr{P}\subsetneqq X_\mathscr{P}\setminus\{x\}= int_XZ(h')\cap X_\mathscr{P}$. So, $h'\in V(C_\mathscr{P})\setminus\{f,g\}$ is not adjacent to $f,g$, leading to a contradiction.\\
Hence, $f,g$ are adjacent in $G(C_\mathscr{P}(X))$ so that 
\begin{equation}\label{eq2}
			\begin{split}
				int_XZ(f)\cap X_\mathscr{P}\not\subset int_XZ(g)\cap X_\mathscr{P}\text{ and }\\
				int_XZ(g)\cap X_\mathscr{P}\not\subset int_XZ(f)\cap X_\mathscr{P}
			\end{split}
		\end{equation}
Our next claim is $f.g = 0$. If possible let $f.g\neq 0$ and choose $x\in X\setminus Z(f)\cap X\setminus Z(g)$. By Lemma~\ref{Lem1}, there exists $h\in C_\mathscr{P}(X)$ such that $x\in X\setminus Z(h)\subset  X\setminus int_XZ(h)\subset X\setminus Z(f)\cap X\setminus Z(g)\implies h\neq 0$ and $Z(f)\cup Z(g)\subset int_XZ(h)$. Clearly, $h'\in V(C_\mathscr{P}(X))$ with $[h']=[h]$, is not adjacent to both $f,g$ in $G(C_\mathscr{P}(X))$. By Equation \ref{eq2} and the fact that $Z(f)\cup Z(g)\subset int_XZ(h')$, we get $h'\notin \{f,g\}$, which contradicts that $\{f,g\}$ is a dominating set in $G(C_\mathscr{P}(X))$.\\
		Finally, to show that $int_XZ(f)\cap int_XZ(g)\cap X_\mathscr{P}=\emptyset$. If possible let $int_XZ(f)\cap int_XZ(g)\cap X_\mathscr{P}\neq\emptyset$ and choose $x\in int_XZ(f)\cap int_XZ(g)\cap X_\mathscr{P}$. Let $h=f^2+g^2\in C_\mathscr{P}(X)$. Then $int_XZ(h)=int_XZ(f)\cap int_XZ(g)\implies int_XZ(h)\cap X_\mathscr{P}\neq\emptyset\implies h\in Z(C_\mathscr{P}(X))^*$. If $h'\in V(C_\mathscr{P}(X))$ is such that $[h']=[h]$ then $int_XZ(h')=int_XZ(f)\cap int_XZ(g)$. So, by Equation \ref{eq2}, $h'\neq f$, $h'\neq g$. Since $int_XZ(f)\cap int_XZ(g)\cap X_\mathscr{P}=int_XZ(h')\cap X_\mathscr{P}$, $h'$ is not adjacent to both $f,g$ in $G(C_\mathscr{P}(X))$ and we arrive at a contradiction. 
	\end{proof}
An immediate consequence is recorded in the following corollary.
	\begin{corollary}\label{CorTT}
		$dt(G(C_\mathscr{P}(X)))=2$ if and only if $dt_t(G(C_\mathscr{P}(X)))=2$.
	\end{corollary}
\begin{lemma}\label{LemT}
		Every total dominating set in $G(C_\mathscr{P}(X))$ is also a total dominating set in $AG(C_\mathscr{P}(X))$.
	\end{lemma}
	\begin{proof}
		Let $W$ be a total dominating set in $G(C_\mathscr{P}(X))$ and $f\in Z(C_\mathscr{P}(X))^*$. Then there exists $f'\in V(C_\mathscr{P}(X))$ such that $[f']=[f]$. Since $W$ is a total dominating set in $G(C_\mathscr{P}(X))$, there exists $g\in W$ such that $f',g$ are adjacent in $G(C_\mathscr{P}(X))$ and hence $f$ and $g$ are adjacent in $AG(C_\mathscr{P}(X))$. Therefore, $W$ is a total dominating set in $AG(C_\mathscr{P}(X))$.
	\end{proof}
	\begin{theorem}\label{ThmT}
		$dt_t(G(C_\mathscr{P}(X)))=dt_t(AG(C_\mathscr{P}(X)))$.
	\end{theorem}
	\begin{proof}
		Let $H$ be a total dominating set in $AG(C_\mathscr{P}(X))$. Consider $H'=\{f'\in V(C_\mathscr{P}(X)):[f']=[f]\text{ for some }f\in H\}$. Let $f\in V(C_\mathscr{P}(X))$. Then $f\in Z(C_\mathscr{P}(X))^*$ and hence there exists $g\in H$ such that $f,g$ are adjacent in $AG(C_\mathscr{P}(X))$. Therefore $f$ and $g'$ are adjacent in $G(C_\mathscr{P}(X))$ where $g'\in H'$. It follows that $H'$ is a total dominating set in $AG(C_\mathscr{P}(X))$. Therefore, $dt_t(G(C_\mathscr{P}(X)))\leq |H'|\leq |H|$, proving $dt_t(G(C_\mathscr{P}(X)))\leq dt_t(AG(C_\mathscr{P}(X)))$. Conversely, let $W$ be a total dominating set in $G(C_\mathscr{P}(X))$. Then by Lemma~\ref{LemT}, $W$ is a total dominating set in $AG(C_\mathscr{P}(X))$ and hence $dt_t(AG(C_\mathscr{P}(X)))\leq |W|$. Hence, $dt_t(AG(C_\mathscr{P}(X)))\leq dt_t(G(C_\mathscr{P}(X)))$
	\end{proof}
	The next corollary follows from Theorem~\ref{Th7.16} and Theorem~\ref{ThmT}
	\begin{corollary}
		$dt(G(C_\mathscr{P}(X)))=dt_t(G(C_\mathscr{P}(X)))$ if and only if $dt(G(C_\mathscr{P}(X)))=dt(AG(C_\mathscr{P}(X)))$ $=dt_t(AG(C_\mathscr{P}(X)))=dt_t(G(C_\mathscr{P}(X)))$.
	\end{corollary}
	We cite an example where $dt(G(C_\mathscr{P}(X)))= dt_t(G(C_\mathscr{P}(X)))$.
	\begin{example}
		Let $X$ be a completely regular Hausdorff space with atleast two isolated points and $\mathscr{P}$ be the ideal of all finite sets of $X$. Then $C_\mathscr{P}(X)=C_F(X)$ and $X_\mathscr{P}$ is the set of all isolated points of $X$. So, $|X_\mathscr{P}|\geq 2$. As seen before, $V(C_\mathscr{P}(X))=\{1_A:\emptyset\neq A\subsetneqq X_\mathscr{P}\}$.\\
		If $X_\mathscr{P}$ is finite then $ X_\mathscr{P}$ is clopen and so, proceeding as in Theorem~\ref{Th7.19}, $W'=\{1_x,1_{X_\mathscr{P}\setminus\{x\}}\}$, for some $x\in X_\mathscr{P}$ is a total dominating set of $G(C_\mathscr{P}(X))$. So $dt_t(G(C_\mathscr{P}(X)))=2$ and hence $dt(G(C_\mathscr{P}(X)))= dt_t(G(C_\mathscr{P}(X)))$, by Corollary~\ref{CorTT}.\\
		If $X_\mathscr{P}$ is infinite then set $W'=\{1_x:x\in K\}$, for some countable infinite subset $K$ of $X_\mathscr{P}$. Proceeding as in Theorem~\ref{dom}, it is easy to see that, $W'$ is a total dominating set in $G(C_\mathscr{P}(X))$. Clearly, $dt_t(G(C_\mathscr{P}(X)))\leq |W'|=|K|=\aleph_0$. If $W$ is any dominating set of $G(C_\mathscr{P}(X))$ then it is enough to show that, $W$ is never finite. If possible let $W=\{f_i\in V(C_\mathscr{P}(X)):i=1,2,\ldots,n\}$. Choosing any $x\in X_\mathscr{P}\setminus(\bigcup\limits_{i=1}^n (X\setminus Z(f_i)))$ (such a point exists as $X_\mathscr{P}$ is infinite),  we construct a finite subset $A=\{x\} \cup \bigcup\limits_{i=1}^n (X\setminus Z(f_i))$ of $X_\mathscr{P}$. Then $X\setminus Z(f_i)\subsetneqq A=X\setminus Z(1_A)$ for each $i=1,2,\ldots,n$. So $1_A\in V(C_\mathscr{P}(X))\setminus W$ and $1_A$ is not adjacent to any $f_i$, $i=1,2,...,n$, contradicting the assumption that $W$ is a dominating set. Thus $dt(G(C_\mathscr{P}(X)))\geq \aleph_0\geq dt_t(G(C_\mathscr{P}(X)))$. Hence $dt(G(C_\mathscr{P}(X))) =dt_t(G(C_\mathscr{P}(X)))$.
	\end{example}
So, from Theorem~\ref{dom} and the last example, we get
$$dt(G(C_F(X)))=\begin{cases}
			2&\text{when }X\text{ has finitely many isolated points }\\
			\aleph_0&\text{otherwise}
		\end{cases}$$

	\begin{theorem}\label{Th7.17}
		$cl(G(C_\mathscr{P}(X)))=cl(AG(C_\mathscr{P}(X)))$.
	\end{theorem}
\begin{proof}
As $G(C_\mathscr{P}(X))$ is a subgraph $AG(C_\mathscr{P}(X))$, $cl(G(C_\mathscr{P}(X)))\leq cl(AG(C_\mathscr{P}(X)))$.\\
Let $H$ be any complete subgraph of $AG(C_\mathscr{P}(X))$. Consider $H'=\{f'\in V(C_\mathscr{P}(X)):[f']=[f]\text{ for some }f\in H\}$. Completeness of $H'$ is immediate from the completeness of $H$. Also, for any $f \in AG(C_\mathscr{P}(X))$, there exists a unique $f'\in G(C_\mathscr{P}(X))$ with $[f] = [f']$ defines a function $\psi : H \rightarrow H'$ given by $f \mapsto f'$ whenever $[f] = [f']$. As $H$ is complete, it follows that $\psi$ is one-one. The construction of $H'$ itself indicates that $\psi$ is onto. \\
So, $|H| = |H'|\leq cl(G(C_\mathscr{P}(X)))$ implies that $cl(AG(C_\mathscr{P}(X)))\leq cl(G(C_\mathscr{P}(X)))$. 
\end{proof}

\begin{theorem}\label{Th7.18}
		$\chi (G(C_\mathscr{P}(X)))=\chi (AG(C_\mathscr{P}(X)))$.
	\end{theorem}
	\begin{proof}
		Certainly,  $\chi (G(C_\mathscr{P}(X)))\leq \chi (AG(C_\mathscr{P}(X)))$.\\
Let $\chi(G(C_\mathscr{P}(X))) = \Lambda$. If $f\in Z(C_\mathscr{P}(X))^*$ is such that $f\in V(C_\mathscr{P}(X))$ then it is already colored. So, we assume $f\notin V(C_\mathscr{P}(X))$. For each $f$ there exists a unique $f'\in V(C_\mathscr{P}(X))$ such that $[f']=[f]$. As $f,f'$ are non-adjacent in $AG(C_\mathscr{P}(X))$, we color $f$ by the color of $f'$. We claim that the coloring is consistent. If $f, g$ in $AG(C_\mathscr{P}(X))$ have the same color, say $f'$, then by the rule of coloring, $[f] = [f'] = [g]$ which implies that $f$ and $g$ are non-adjacent, proving consistency, as desired.
So, $\chi (AG(C_\mathscr{P}(X)))\leq \chi (G(C_\mathscr{P}(X)))$.
	\end{proof}

\section{An algorithm for coloring $AG(C_F(X))$}

In this section, we show that the chromatic number and the clique number of $AG(C_F(X))$ are same and this value is completely determined by the cardinality of the set of all isolated points of $X$. Moreover, if $|K_X|$ is finite then $G(C_F(X))$ is a finite graph and we are successful in devising an algorithm to color the vertices of $G(C_F(X))$. As a consequence, obtain the chromatic number of $AG(C_F(X))$, even though it is an infinite graph.

\begin{theorem}\label{Th7.21}
If $|K_X|$ is infinite then $|V(C_F(X))| = |K_X|$. 
\end{theorem}
\begin{proof}
From Theorem~\ref{Th7.3}, taking $\mathscr{P}$ as the ideal of all finite subsets of $X$, it follows that $|V(C_F(X))|$ is infinite. Clearly, $|K_X| = |\{1_x : x\in K_X\}| \leq |V(C_F(X))|$.  $\{X\setminus Z(f): f\in V(C_F(X))\}$ being a family of nonempty finite subsets of $K_X$, $\{X\setminus Z(f): f\in V(C_F(X))\}\subset \bigcup\limits_{n=1}^\infty (K_X)^n =|K_X|$ (as $|K_X|$ is infinite) so that $|\{X\setminus Z(f): f\in V(C_F(X))\}|\leq |K_X|$. For distinct $f,g\in V(C_F(X))$, $X\setminus Z(f)\neq X\setminus Z(g)\implies |V(C_F(X))|=|\{X\setminus Z(f): f\in V(C_F(X))\}|$, i.e., $|V(C_F(X))|\leq |K_X|$. 
\end{proof}
\begin{theorem}\label{Th7.22}
If $|K_X|$ is infinite then $cl(G(C_F(X))) = |K_X| = \chi(G(C_F(X)))$.
\end{theorem}
\begin{proof}
Since $cl(G(C_F(X)))\leq |V(C_F(X))|$ it follows from Theorem~\ref{Th7.21} that $cl(G(C_F(X)))\leq |K_X|$. Observing that the subgraph $H$ of $G(C_F(X))$ whose set of vertices is $\{1_x:x\in K_X\}$, is complete, we get $cl(G(C_F(X)))\geq |H|=|K_X|$. Hence, $cl(G(C_F(X)))=|K_X|$. \\
Also, $cl(G(C_F(X)))\leq \chi(G(C_F(X)))\leq V(C_F(X))$. By Theorem~\ref{Th7.21} the result follows.
\end{proof}

\begin{theorem}\label{ThClique}
If $|K_X|$ is finite then $cl(G(C_F(X)))\geq \binom{|K_X|}{[\frac{|K_X|}{2}]}$, where $\binom{n}{r} = \frac{n!}{(n-r)! r!}$  and $[m]$ denotes the positive integer less than or equal to $m$.
\end{theorem}
\begin{proof}
Let $|K_X| = \{x_1, x_2, \ldots, x_n\}$. We define an order `$<$' on $K_X$, so that $x_1 < x_2 < \ldots < x_n$. For each $k\leq n$, define $\mathscr{A}_k = \{A \subset K_X : |A| = k\}$.\\
We make a convention that for any $A\in \mathscr{A}_k$, whenever we write $A = \{a_1, a_2, \ldots, a_k\}$, it is understood that $a_1<a_2<\ldots < a_k$.\\
Define an order $<$ on $\mathscr{A}_k$ by $A < B$ if and only if either ($a_1 < b_1$) or (there exists some $i \in \{2,\ldots, n\}$ such that $a_i <b_i$ and $a_j = b_j$ for all $j < i$).
It is clear that for each $k=1, 2, \ldots, n$, $(\mathscr{A}_k, <)$ is a totally ordered set with the least element $A^1 = \{x_1, x_2, \ldots, x_k\}$. For each $k$, $H_k$, with set of vertices $V(H_k) = \{1_A : A\in \mathscr{A}_k\}$ is a complete subgraph of $G(C_F(X))$. It is further observed that $|\mathscr{A}_k| = \binom{|K_X|}{k}$. Therefore, $|H_k| = \binom{|K_X|}{k}$, for each $k=1, 2, \ldots, n$. \\
Among the complete subgraphs $H_k$ of $G(C_F(X))$, $H_{[\frac{|K_X|}{2}]}$ is maximal with number of vertices = $\binom{|K_X|}{[\frac{|K_X|}{2}]}$. Hence,  $cl(G(C_F(X)))\geq \binom{|K_X|}{[\frac{|K_X|}{2}]}$.
\end{proof}
\begin{theorem}\label{ThChrom}
If $|K_X|$ is finite then $\chi(G(C_F(X)))= \binom{|K_X|}{[\frac{|K_X|}{2}]}$.
\end{theorem}
\begin{proof}
Let $|K_x| = n$. Proceeding as in Theorem \ref{ThClique}, we obtain the complete subgraphs $H_k$ of $G(C_F(X))$, among which $H_{[\frac{n}{2}]}$ is maximal with number of vertices $=  \binom{n}{[\frac{n}{2}]}$. So, we start with coloring $H_{[\frac{n}{2}]}$ by $\binom{n}{[\frac{n}{2}]}$-many distinct colors.\\
Before writing the coloring scheme, we make a few conventions and introduce some notations that will be used in the algorithm for coloring the vertices of $G(C_F(X))$.
\begin{enumerate}
\item $A^1$ denotes the first element of each $\mathscr{A}_k$. So, $A^1 = \{x_1, x_2, \ldots, x_k\}$. (From the context we understand which $\mathscr{A}_k$ is under consideration).
\item $A^l$ denotes the last element of each $\mathscr{A}_k$. So, $A^l = \{x_{n-k+1}, x_{n-k+2}, \ldots, x_n\}$. 
\item $curr$ is a variable, used to denote some member of $\mathscr{A}_k$, for any $k$. 
\item $next$ is a variable, used to denote the immediate successor of $curr$.
\item $prev$ is a variable, used to denote a subset of $\mathscr{A}_k$, for any $k$.
\item $1_A \sim 1_B$ reads as ``Color $1_A$ using the color of $1_B$''.
\item $x \leftarrow y$ means ``Assign the value $y$ to the variable $x$''.
\item Since the elements of $A$ are ordered, for any $A \in \mathscr{A}_k$, $A_{-i}$ designates the set obtained from $A$ by deleting its $i$-th element. Clearly, $A \in \mathscr{A}_k$ implies that $A_{-i} \in \mathscr{A}_{k-1}$.
\item For any $A \in \mathscr{A}_{k-1}$, define $A_{+j} = A \cup \{x_j\}$. It is to note that $ A_{+j}\in \mathscr{A}_k$ if and only if $x_j \notin A$.  
\end{enumerate}
It is quite clear that if $A^1$ is the first element of $\mathscr{A}_k$ then ${A^1}_{-k}$ is the first element of $\mathscr{A}_{k-1}$.\\
\textbf{Step 1:} Starting from $k= [\frac{n}{2}]$, we color $H_{k-1}$ by using the colors of $H_k$, adopting the following algorithm:
\begin{enumerate}
\item[1.] $k \leftarrow [\frac{n}{2}]-1$
\item[2.] $curr \leftarrow A^1$
\item[3.] $1_{curr} \sim 1_{curr_{+k}}$
\item[4.] $prev\leftarrow \{curr_{+k}\}$
\item[5.] $curr \leftarrow next$
\item[6.] IF $j \in \{1, 2, \ldots, n\}$ is the smallest such that $curr_{+j}\in \mathscr{A}_k$ and $curr_{+j} \neq B$, for all $B \in prev$ THEN $1_{curr} \sim 1_{curr_{+j}}$ and $prev \leftarrow prev \cup \{curr_{+j}\}$
\item[7.] IF $curr = A^l$ THEN GOTO Step 8 ELSE GOTO Step 5
\item[8.] $k\leftarrow k-1$
\item[9.] CONTINUE  Step 2 to Step 8 TILL $k = 1$
\item[10.] STOP when $k=1$.
\end{enumerate}
We first check that the scheme described above exhausts coloring of all the vertices of $H_{k-1}$ with the colors used for $H_k$. The first element of $H_{k-1}$ is colored by the first element of the $H_k$. It is enough to show that the last element of $H_{k-1}$ gets a color.\\
We first show that $1_A \sim 1_B$ where $A = \{x_1, x_2, \ldots, x_{k-3}, x_{n-1}, x_n\}$ and $B =$ $\{x_1, x_2, \ldots,$ $x_{k-3}, x_{k-2}, x_{n-1}, x_n\}$.\\
Using the scheme of coloring described above, it is not hard to see that the vertex corresponding to $\{x_1, x_2, \ldots, x_{k-3}, x_{k-2}, x_{n-1}\}$ is colored by the vertex corresponding to $\{x_1, x_2, \ldots, x_{k-2}, x_{k-1}, x_{n-1}\}$ and $\{x_1, x_2, \ldots,$ $ x_{k-3}, x_{k-2}, x_{n}\}$ by $\{x_1, x_2, \ldots,$ $x_{k-2}, x_{k-1}, x_{n}\}$. Also, $B_{-i} = \{x_1, x_2, \ldots, x_{i-1}, x_{i+1}, \ldots, x_{n-1}, x_n\} > A$ as $x_{i+1} > x_i$. Hence, $1_{B_{-i}}$ ($i=1, 2, \ldots, k-2$) are not colored by $1_B$ and therefore, $1_B$ is available for $1_A$.\\
Proceeding similarly, we get that $1_A \sim 1_B$ where $A = \{x_1, x_{n-k+3}, \ldots, x_{n-1}, x_n\}$ and $B = \{x_1, x_2, x_{n-k+3},\ldots, x_{n-1}, x_n\}$.\\
So, for $A^l = \{x_{n-k+2}, x_{n-k+3}, \ldots, x_{n-1}, x_n\}$, $C = \{x_1, x_{n-k+2}, \ldots, x_{n-1}, x_n\}$ is available, so that $1_{A^l} \sim 1_C$ (because, none of $1_{C_{-i}}$ ($i \neq 1$) are colored by $C$, as seen above). \\
\textbf{Step 2:}
Starting from $k= [\frac{n}{2}]$, we color $H_{k+1}$ by using the colors of $H_k$, adopting the following algorithm:
\begin{enumerate}
\item[1.] $k \leftarrow [\frac{n}{2}] + 1$
\item[2.] $curr \leftarrow A^1$
\item[3.] $ 1_{curr} \sim 1_{curr_{-k}}$
\item[4.] $prev \leftarrow \{curr_{-k}\}$
\item[5.] $curr \leftarrow next$
\item[6.] IF $j$ is the largest such that $curr_{-j} \neq B$, for all $B \in prev$ THEN $ 1_{curr}\sim 1_{curr_{-j}}$ and $prev \leftarrow prev \cup \{curr_{-j}\}$
\item[7.] IF $curr =A^l$ THEN GOTO Step 8 ELSE GOTO Step 5
\item[8.] $k\leftarrow k+1$
\item[9.] CONTINUE  Step 2 TO Step 8 TILL $k = n$
\item[10.] STOP when $k=n$.
\end{enumerate}
We claim that the colors (i.e., vertices of $H_{k-1}$) do not exhaust before all the vertices of $H_k$ are colored. Let $A \in \mathscr{A}_{k}$ and for each $j$, $A_{-j} \in prev$. Let $B^j \in \mathscr{A}_k$ be such that $1_{B^j} \sim 1_{A_{-j}}$, for $j =1, 2, \ldots k$. Also, each $B^j$ has been colored before $A$ implies that $B^j < A$, for all $j =1, 2,\ldots k$. If $A = \{a_1, a_2, \ldots, a_k\}$ then $B^j = \{a_1, \ldots, a_{j-1}, a_{j+1}, \ldots, a_k\} \cup \{b_j\}$ where $b_j \in K_X$. Clearly, $b_j \neq a_i$, for each $i\neq j$ and $b_j \neq a_j$ as $A \neq B^j$.  Hence, there are atleast $2k$ distinct elements (i.e., $a_1, \ldots, a_k, b_1, \ldots, b_k$) exist in $K_X$. But, $[\frac{n}{2}]+1$ being the minimum value of $k$ in the above algorithm, $2k \geq 2([\frac{n}{2}]+1) > n = |K_X|$, which is impossible. Therefore, using a color of $H_{k-1}$ it is possible to color all the vertices of $H_{k}$, for each $k = [\frac{n}{2}] + 1$, \ldots, n. \\ 
Proceeding as in the previous case, we observe that if $1_A \in H_i$ and $1_B \in H_j$ ($i< j$) have the same color then (as $i< j$) $A\subsetneqq B$.\\
Finally, we check whether the coloring is consistent. For that matter, we choose any $1_A$ and $1_B$ from $V(C_F(X))$ such that they have the same color. Then $1_A \in H_i$ and $1_B \in H_j$ for some $i \neq j$. In other words, $A \in \mathscr{A}_i$ and $B \in \mathscr{A}_j$. Without loss of generality, we assume that $i < [\frac{n}{2}]$ and $j > [\frac{n}{2}]$. Then proceeding as before, we get some $1_C, 1_D \in H_{[\frac{n}{2}]}$ such that $1_A$ and $1_C$ have the same color and $1_B$ and $1_D$ have the same color. So, $A \subsetneqq C$ and $D \subsetneqq B$. As $1_A$ and $1_B$ have the same color, it follows that $1_C$ and $1_D$ also have the same color, which is possible when and only when $C = D$. Therefore, $A \subsetneqq C=D \subsetneqq B$; i.e., $1_A$ and $1_B$ are non-adjacent.
\end{proof}
\begin{corollary}\label{GFinite}
If $|K_X|$ is finite, $ cl(G(C_F(X))) = \chi(G(C_F(X)))= \binom{|K_X|}{[\frac{|K_X|}{2}]}$.
\end{corollary}
\begin{proof}
Since $ cl(G(C_F(X))) \leq \chi(G(C_F(X)))$, the result follows from Theorem~\ref{ThChrom} and Theorem \ref{ThClique}.
\end{proof}
The following example illustrates the coloring scheme of Theorem~\ref{ThChrom}.
\begin{example}
Let $X$ be a Tychonoff space with $K_X=\{x_1, x_2, x_3, x_4, x_5\}$. Then $|K_X| = 5$.
\begin{center}
	\begin{tikzpicture}
		\foreach \x in {-0.25,1.5,3.25,5,6.75,8.5,10.25,12,13.75,15.5} \draw (\x,0) circle (0.2);
		\foreach \x in {0,1.5,3,4.5,6,7.5,9,10.5,12,13.5} \draw (\x,1.5) circle (0.2);
		\foreach \x in {0.75,2.25,3.75,5.25,6.75} \draw (\x,3) circle (0.2);
		\foreach \x in {0.75,3.75,6.75,9.75,12.75} \draw (\x,-1.5) circle (0.2);
		\node at (-1.5,0) {$H_3 :$};
		\node at (-1.5,1.5) {$H_2 :$};
		\node at (-1.5,3) {$H_1 :$};
		\node at (-1.5,-1.5) {$H_4 :$};
		\node at (0.75, 2.5) {$1_{x_1}$};
		\node at (2.25, 2.5) {$1_{x_2}$};
		\node at (3.75, 2.5) {$1_{x_3}$};
		\node at (5.25, 2.5) {$1_{x_4}$};
		\node at (6.75, 2.5) {$1_{x_5}$};
		\node at (0, 1) {$1_{\{x_1,x_2\}}$};
		\node at (1.5, 1) {$1_{\{x_1,x_3\}}$};
		\node at (3, 1) {$1_{\{x_1,x_4\}}$};
		\node at (4.5, 1) {$1_{\{x_1,x_5\}}$};
		\node at (6, 1) {$1_{\{x_2,x_3\}}$};
		\node at (7.5, 1) {$1_{\{x_2,x_4\}}$};
		\node at (9, 1) {$1_{\{x_2,x_5\}}$};
		\node at (10.5, 1) {$1_{\{x_3,x_4\}}$};
		\node at (12, 1) {$1_{\{x_3,x_5\}}$};
		\node at (13.5, 1) {$1_{\{x_4,x_5\}}$};
		\node at (-0.25, -0.5) {$1_{\{x_1,x_2,x_3\}}$};
		\node at (1.5, -0.5) {$1_{\{x_1,x_2,x_4\}}$};
		\node at (3.25, -0.5) {$1_{\{x_1,x_2,x_5\}}$};
		\node at (5, -0.5) {$1_{\{x_1,x_3,x_4\}}$};
		\node at (6.75, -0.5) {$1_{\{x_1,x_3,x_5\}}$};
		\node at (8.5, -0.5) {$1_{\{x_1,x_4,x_5\}}$};
		\node at (10.25, -0.5) {$1_{\{x_2,x_3,x_4\}}$};
		\node at (12, -0.5) {$1_{\{x_2,x_3,x_5\}}$};
		\node at (13.75, -0.5) {$1_{\{x_2,x_4,x_5\}}$};
		\node at (15.5, -0.5) {$1_{\{x_3,x_4,x_5\}}$};
		\node at (0.75, -2) {$1_{\{x_1,x_2,x_3,x_4\}}$};
		\node at (3.75, -2) {$1_{\{x_1,x_2,x_3,x_5\}}$};
		\node at (6.75, -2) {$1_{\{x_1,x_2,x_4,x_5\}}$};
		\node at (9.75, -2) {$1_{\{x_1,x_3,x_4,x_5\}}$};
		\node at (12.75, -2) {$1_{\{x_2,x_3,x_4,x_5\}}$};
		\fill[fill=black] (0,1.5) circle (0.2);
		\fill[fill=blue] (1.5,1.5) circle (0.2);
		\fill[fill=yellow] (3,1.5) circle (0.2);
		\fill[fill=green] (4.5,1.5) circle (0.2);
		\fill[fill=red] (6,1.5) circle (0.2);
		\fill[fill=brown] (7.5,1.5) circle (0.2);
		\fill[fill=orange] (9,1.5) circle (0.2);
		\fill[fill=pink] (10.5,1.5) circle (0.2);
		\fill[fill=gray] (12,1.5) circle (0.2);
		\fill[fill=lime] (13.5,1.5) circle (0.2);
		\fill[fill=black] (-0.25,0) circle (0.2);
		\fill[fill=yellow] (1.5,0) circle (0.2);
		\fill[fill=green] (3.25,0) circle (0.2);
		\fill[fill=blue] (5,0) circle (0.2);
		\fill[fill=gray] (6.75,0) circle (0.2);
		\fill[fill=lime] (8.5,0) circle (0.2);
		\fill[fill=red] (10.25,0) circle (0.2);
		\fill[fill=orange] (12,0) circle (0.2);
		\fill[fill=brown] (13.75,0) circle (0.2);
		\fill[fill=pink] (15.5,0) circle (0.2);
		\fill[fill=black] (0.75,3) circle (0.2);
		\fill[fill=red] (2.25,3) circle (0.2);
		\fill[fill=blue] (3.75,3) circle (0.2);
		\fill[fill=yellow] (5.25,3) circle (0.2);
		\fill[fill=green] (6.75,3) circle (0.2);
		\fill[fill=black] (0.75,-1.5) circle (0.2);
		\fill[fill=green] (3.75,-1.5) circle (0.2);
		\fill[fill=yellow] (6.75,-1.5) circle (0.2);
		\fill[fill=blue] (9.75,-1.5) circle (0.2);
		\fill[fill=red] (12.75,-1.5) circle (0.2);
	\end{tikzpicture}
\end{center}
\end{example}

Combining Theorem~\ref{Th7.17}, Theorem~\ref{Th7.18}, Theorem~\ref{Th7.22} and Corollary~\ref{GFinite}, we get 
\begin{theorem}
$$\chi(AG(C_F(X)))= cl(AG(C_F(X))) = \begin{cases} \binom{|K_X|}{[\frac{|K_X|}{2}]}, &  \textnormal{ if } |K_X| \textnormal{ is finite} \cr |K_X|, & \textnormal{ otherwise}\end{cases}$$
\end{theorem}

\section{On the induced graph isomorphisms}

In the earlier sections, we have seen that $G(C_\mathscr{P}(X))$ and $AG(C_\mathscr{P}(X))$ behave in a similar fashion so far as triangulatedness, hypertriangulatedness and complementedness are concerned and both have the same diameter, girth, eccentricity, clique number and chromatic number. In this section, we investigate whether a graph isomorphism $G(C_\mathscr{P}(X))\rightarrow G(C_\mathscr{Q}(Y))$ completely determines a graph isomorphism $AG(C_\mathscr{P}(X)) \rightarrow AG(C_\mathscr{Q}(Y))$ and get a partial answer to this query. However, we show that for $C_\mathscr{P}(X)$ with $|X_\mathscr{P}| =$ finite, the answer is complete. \\

In what follows, $X,Y$ stand for two completely regular Hausdorff spaces and $\mathscr{P}, \mathscr{Q}$ are ideals of closed sets in $X$ and $Y$ respectively. The equivalence relations on $C_\mathscr{P}(X)$ and $C_\mathscr{Q}(Y)$ are denoted respectively by $\sim_X$, $\sim_Y$ and their corresponding equivalence classes by $[f]_X$ and $[g]_Y$, ($f \in C_\mathscr{P}(X)$, $g \in C_\mathscr{Q}(Y))$ respectively. \\

We first show that the restriction of an isomorphism $\psi : AG(C_\mathscr{P}(X)) \rightarrow AG(C_\mathscr{Q}(Y))$ on $G(C_\mathscr{P}(X))$ takes $G(C_\mathscr{P}(X))$ isomorphically onto $G(C_\mathscr{Q}(Y))$. Before that we need a lemma : 

\begin{lemma}\label{Lem6.1}
		Let $\psi: AG(C_\mathscr{P}(X))\to AG(C_\mathscr{Q}(Y))$ be a graph isomorphism and $f,g\in Z(C_\mathscr{P}(X))^*$. Then $int_XZ(f)\cap X_\mathscr{P}=int_XZ(g)\cap X_\mathscr{P}$ if and only if $int_YZ(\psi(f))\cap Y_\mathscr{Q}=int_YZ(\psi(g))\cap Y_\mathscr{Q}$.
	\end{lemma}
	\begin{proof}
		Let $int_XZ(f)\cap X_\mathscr{P}=int_XZ(g)\cap X_\mathscr{P}$, for some $f, g \in AG(C_\mathscr{P}(X))$. Then $f,g$ are not adjacent in $AG(C_\mathscr{P}(X))$ and so, $\psi(f), \psi(g)$ are not adjacent in $AG(C_\mathscr{Q}(Y))$. Without loss of generality, assume that $int_YZ(\psi(f))\cap Y_\mathscr{Q}\subset int_YZ(\psi(g))\cap Y_\mathscr{Q}$.\\
For if $y \in [int_YZ(\psi(g))\cap Y_\mathscr{Q}]\setminus Z(\psi(f))$, by Lemma~\ref{Lem1}, there exists $h\in C_\mathscr{Q}(Y)$ such that $y\in Y\setminus Z(h)\subset Y\setminus int_YZ(h)\subset [int_YZ(\psi(g))\cap Y_\mathscr{Q}]\setminus Z(\psi(f))$. Clearly, $h\neq 0$, $Z(\psi(f))\subset int_YZ(h)$ and $\psi(g)h=0$. Now it is a routine check that $\psi(g)h = 0 \implies \psi(f)$ and $h$ are adjacent and hence, contradicts the assumption. 
So, $ int_YZ(\psi(g))\cap Y_\mathscr{Q}\subset Z(\psi(f))\cap Y_\mathscr{Q}$.\\
The converse part is clear.
\end{proof}

	\begin{theorem}\label{Th6.2}
		If $\psi : AG(C_\mathscr{P}(X)) \to AG(C_\mathscr{Q}(Y))$ is a graph isomorphism then $\psi|_{G(C_\mathscr{P}(X))}$ maps $G(C_\mathscr{P}(X))$ isomorphically onto $G(C_\mathscr{Q}(Y))$.
	\end{theorem}
	\begin{proof}
		Let $\psi: AG(C_\mathscr{P}(X))\to AG(C_\mathscr{Q}(Y))$ be a graph isomorphism and $f_X\in V(C_\mathscr{P}(X))$. Then $\psi(f_X) \in Z(C_\mathscr{Q}(Y))^*$. By Lemma~\ref{Lem5.1}, there exists a unique $f_Y\in V(C_\mathscr{Q}(Y))$ such that $[f_Y]_Y=[\psi(f_X)]_Y$; i.e., $int_YZ(f_Y)\cap Y_\mathscr{Q}=int_YZ(\psi(f_X))\cap Y_\mathscr{Q}$. Consider the map $\phi: V(C_\mathscr{P}(X))\to V(C_\mathscr{Q}(Y))$ given by $\phi(f_X)=f_Y$, where $f_Y\in V(C_\mathscr{Q}(Y))$ such that $int_YZ(f_Y)\cap Y_\mathscr{Q}=int_YZ(\psi(f_X))\cap Y_\mathscr{Q}$.\\
		If $f_Y\in V(C_\mathscr{Q}(Y))$ then $\psi^{-1}(f_Y)\in Z(C_\mathscr{P}(X))^*$. Choosing $f_X\in V(C_\mathscr{P}(X))$ in such a way that $int_XZ(f_X)\cap X_\mathscr{P}=int_XZ(\psi^{-1}(f_Y))\cap X_\mathscr{P}$. By Lemma~\ref{Lem6.1}, it is easy to see that $\phi(f_X)=f_Y$.  As a result, $\phi$ is a bijective map and it is easy to see that $\phi:G(C_\mathscr{P}(X))\to G(C_\mathscr{Q}(Y))$ is a graph isomorphism.
	\end{proof}

\begin{theorem}\label{Th6.3}
		Let $\phi:G(C_\mathscr{P}(X))\to G(C_\mathscr{Q}(Y))$ be a graph isomorphism satisfying the condition $|[f]_X| = |[\phi(f)]_Y|$, for all $f\in V(C_\mathscr{P}(X))$ then $\phi$ can be extended to a graph isomorphism between $AG(C_\mathscr{P}(X))$ and $AG(C_\mathscr{Q}(Y))$.
	\end{theorem}
	\begin{proof}
		For each $f\in Z(C_\mathscr{P}(X))^*$ there exists $f'\in V(C_\mathscr{P}(X))$ such that $[f']_X=[f]_X$. So, $\phi(f')\in V(C_\mathscr{Q}(Y))$ and hence by the hypothesis, $|[f']_X|=|[\phi(f')]_Y|$ which implies that $|[f]_X|=|[\phi(f')]_Y|$ and so, there exists a bijection between $|[f]_X|$ and $|[\phi(f')]_Y|$, say $\psi_{f'}:[f]_X\to [\phi(f')]_Y$. Since an equivalence relation on a set yields a partition, pasting these bijections we get a bijective map $\psi$ from $Z(C_\mathscr{P}(X))^*$ to $Z(C_\mathscr{Q}(Y))^*$. \\
Let $f,g\in Z(C_\mathscr{P}(X))^*$ be adjacent in $AG(C_\mathscr{P}(X))$. Then there are $f',g'\in V(C_\mathscr{P}(X))$ such that $[f']_X=[f]_X$ and $[g']_X=[g]_X$. $f',g'$ are adjacent in $G(C_\mathscr{P}(X))\implies \phi(f'),\phi(g')$ are adjacent in $G(C_\mathscr{Q}(Y))$. Now $\psi(f)=\psi_{f'}(f)\in [\phi(f')]_Y\implies int_YZ(\psi(f))\cap Y_\mathscr{Q}=int_YZ(\phi(f'))\cap Y_\mathscr{Q}$ and $\psi(g)=\psi_{g'}(g)\in [\phi(g')]_Y\implies int_YZ(\psi(g))\cap Y_\mathscr{Q}=int_YZ(\phi(g'))\cap Y_\mathscr{Q}$. It follows that $\psi(f), \psi(g)$ are adjacent in $AG(C_\mathscr{Q}(Y))$. Similarly, if $f, g$ are adjacent in $AG(C_\mathscr{Q}(Y))$ then $\psi^{-1}(f), \psi^{-1}(g)$ are adjacent in $AG(C_\mathscr{P}(X))$ proving that $\psi:AG(C_\mathscr{P}(X))\to AG(C_\mathscr{Q}(Y))$ is a graph isomorphism.
	\end{proof}

The following example guarantees the existence of a ring $C_\mathscr{P}(X)$ obeying the condition of Theorem~\ref{Th6.3}
	\begin{example}
		Let $X,Y$ be two completely regular Hausdorff spaces and $\mathscr{P},\mathscr{Q}$ be two ideal of closed sets on $X,Y$ respectively such that $|X_\mathscr{P}|\geq 2, |Y_\mathscr{Q}|\geq 2$ are finite. (For example, taking $X,Y$ as completely regular Hausdorff spaces with atleast two isolated points and $\mathscr{P}=\{\{x\},\{y\},\{z\},\{x,y\},\{y,z\},\{x,z\},\{x,y,z\}\}$ and $\mathscr{Q}=\{\{a\},\{b\},\{a,b\}\}$, where $x,y,z\in X$ are isolated points in $X$ and $a,b\in Y$ are isolated points in $Y$, we obtain $X_\mathscr{P}=\{x,y,z\}$ and $Y_\mathscr{Q}=\{a,b\}$).\\
If $f\in C_\mathscr{P}(X)$ then $X\setminus Z(f)\subset X_\mathscr{P}\implies X\setminus Z(f)$ is finite and hence, $Z(f)$ is a clopen subset of $X$. So, $X\setminus Z(f)=X_\mathscr{P}\setminus[int_XZ(f)\cap X_\mathscr{P}]$. Let $f\in Z(C_\mathscr{P}(X))^*$. Then $int_XZ(f)\cap X_\mathscr{P}\neq\emptyset\iff X\setminus Z(f)$ is a non-empty proper subset of $X_\mathscr{P}$ and therefore, $Z(C_\mathscr{P}(X))^*=\{f\in C(X):\emptyset\neq X\setminus Z(f)\subsetneqq X_\mathscr{P}\}$.\\
	If $f\in Z(C_\mathscr{P}(X))^*$ and $X\setminus Z(f)=\{x_1,x_2,...x_n\}$ where $1\leq n <|X_\mathscr{P}|$, then $f$ can be written as $f=\sum\limits_{i=1}^nf(x_i)1_{x_i}$. Now, $g\in [f]_X\iff X\setminus Z(g)=X\setminus Z(f)\implies[f]_X=\{\sum\limits_{i=1}^nr_i1_{x_i}:r_i\in\mathbb{R}\setminus\{0\}\text{ for each }i=1,2,...n\}$ and therefore, $|[f]_X|$ comes out as $2^{\aleph_0}$. Thus, for each $f\in Z(C_\mathscr{P}(X))^*$, $|[f]_X|=2^{\aleph_0}$. Similarly, for each $f\in Z(C_\mathscr{Q}(Y))^*$, $|[f]_Y|=2^{\aleph_0}$. If $\phi:G(C_\mathscr{P}(X))\to G(C_\mathscr{Q}(Y))$ is a graph isomorphism then for all $f\in V(\mathscr{P}(X))$, $|[f]_x|=2^{\aleph_0}=|[\phi(f)]_Y|$. Hence, $G(C_\mathscr{P}(X))$ and $G(C_\mathscr{Q}(Y))$ are graph isomorphic $\implies AG(C_\mathscr{P}(X))$ and $AG(C_\mathscr{Q}(Y))$ are graph isomorphic by Theorem~\ref{Th6.3}.
\end{example}

\begin{remark}
		It is also to note that $|X_\mathscr{P}|=|Y_\mathscr{Q}|$ : $\phi$ is a bijection between $V(C_\mathscr{P}(X))$ and $V(C_\mathscr{Q}(Y))\implies|V(C_\mathscr{P}(X))|=|V(C_\mathscr{Q}(Y))|\implies 2^{|X_\mathscr{P}|}-2= 2^{|Y_\mathscr{Q}|}-2$, by Theorem \ref{Th7.3}. So, we find a bijective map between $X_\mathscr{P}$ and $Y_\mathscr{Q}$. The proof of Theorem 6.18 in \cite{Ach} indicates that such a bijective map leads to a ring isomorphism between the rings $C_\mathscr{P}(X)$ and $C_\mathscr{Q}(Y)$.
\end{remark}

Finally, in view of the result that the annihilator graphs of two isomorphic commutative rings are graph isomorphic, we make the following observation:
	\begin{observation}\label{Rem6.5}
		If $X$ and $Y$ are two completely regular Hausdorff spaces and $\mathscr{P}$, $\mathscr{Q}$ are ideals of closed sets on $X, Y$ respectively such that $X_\mathscr{P}$ and $Y_\mathscr{Q}$ are finite sets with $|X_\mathscr{P}|\geq 2$ and $|Y_\mathscr{Q}|\geq 2$ then the following statements are equivalent:
		\begin{enumerate}
			\item $G(C_\mathscr{P}(X))$ and $G(C_\mathscr{Q}(Y))$ are graph isomorphic.
			\item $AG(C_\mathscr{P}(X))$ and $AG(C_\mathscr{Q}(Y))$ are graph isomorphic.
			\item $C_\mathscr{P}(X)$ and $C_\mathscr{Q}(Y)$ are ring isomorphic.
		\end{enumerate}
	\end{observation}

\end{document}